%% file: root.tex
\documentclass[journal]{IEEEtran}
\ifCLASSINFOpdf

\else

\fi
\input{macros}
\usepackage{enumitem}

\allowdisplaybreaks[4]

\newcommand{\RN}[1]{%
\textup{\uppercase\expandafter{\romannumeral#1}}%
}

\newcommand{\changes}{\textcolor{black}}

\begin{document}

\title{Distributed AC Optimal Power Flow: A Scalable Solution for Large-Scale Problems}

\author{Xinliang Dai~\IEEEmembership{Member,~IEEE}, Yuning Jiang~\IEEEmembership{Member,~IEEE}, Yi Guo~\IEEEmembership{Member,~IEEE},\\ 
Colin N. Jones~\IEEEmembership{Senior Member,~IEEE}, Moritz Diehl~\IEEEmembership{Member,~IEEE}, Veit Hagenmeyer~\IEEEmembership{Member,~IEEE}

\thanks{This work was supported in part by the BMBF-project ENSURE III with grant number 03SFK1F0-3, in part by the Swiss National Science Foundation (SNSF) under the NCCR Automation project, grant agreement 51NF40\_180545, and in part by the National Natural Science Foundation of China under grant 62603057, and in part by BMWK via 03EI4057A (GrECCo). (Corresponding authors: Yuning Jiang and Yi Guo)}
\thanks{X.~Dai was with the Institute for Automation and Applied Informatics, Karlsruhe Institute of Technology, Germany. He is now with the Andlinger Center for Energy and the Environment, Princeton University, USA. (e-mail: {\tt xinliang.dai@princeton.edu})}
\thanks{Y. Jiang and Colin N. Jones are with Automatic Control Laboratory, EPFL, Switzerland. Y. Jiang is also with the Institute for Transport Planning and Systems at ETH Zurich, Switzerland. (e-mail: {\tt yuning.jiang@ieee.org, colin.jones@epfl.ch})}
\thanks{Y.~Guo is with the School of Automation, Beijing Institute of Technology, Beijing, China. (e-mail: {\tt yi.guo@bit.edu.cn})}
\thanks{M.~Diehl is with the Department of Microsystems Engineering (IMTEK) and Department of Mathematics, University of Freiburg, Germany. (e-mail: {\tt moritz.diehl@imtek.uni-freiburg.de})}
\thanks{V.~Hagenmeyer is with the Institute for Automation and Applied Informatics, Karlsruhe Institute of Technology, Germany. 
(e-mail: {\tt veit.hagenmeyer@kit.edu})}
}

\maketitle

\begin{abstract}
This paper introduces a novel distributed optimization framework for large-scale AC Optimal Power Flow (OPF) problems, offering both theoretical convergence guarantees and rapid convergence in practice. By integrating smoothing techniques and the Schur complement, the proposed approach addresses the scalability challenges and reduces communication overhead in distributed AC \acrshort{opf}. Additionally, optimal network decomposition enables efficient parallel processing under the \acrfull{spmd} paradigm. 
Extensive simulations on large-scale benchmarks across various operating scenarios indicate that the proposed framework is 2 to 5 times faster than the state-of-the-art centralized solver IPOPT on modest hardware. This paves the way for more scalable and efficient distributed optimization in future power system applications.


\end{abstract}


\IEEEpeerreviewmaketitle

\section{Introduction}\label{sec::Intro}

The AC \acrfull{opf} problem is one of the most practically important optimization problems in electric power systems engineering~\cite{frank2016introduction}. It is generally NP-hard, even for radial power grids~\cite{lehmann2015ac,bienstock2019strong}.  Traditionally, this problem is solved by using centralized approaches, primarily for long-term scheduling and planning purposes~\cite{frank2016introduction}. However, the ongoing energy transition presents new challenges for these centralized paradigms. The increasing integration of \acrfull{ders} introduces rapid fluctuations in energy supply and demand, complicating the management of transmission and distribution systems and requiring enhanced coordination among system operators~\cite{muhlpfordt2021distributed}. Additionally, as more controllable devices are installed in power systems, centralizing all data not only raises significant privacy concerns but also places greater demands on communication infrastructure~\cite{patari2021distributed}. 

In response to these challenges, distributed optimization offers an efficient alternative for coordinating geographically dispersed systems, allowing independent operation and effective collaboration through limited information sharing. Potential advantages of distributed optimization algorithms in power systems include~\cite{muhlpfordt2021distributed,molzahn2017surveyDistr,patari2021distributed}:
\begin{itemize}[leftmargin=15pt]
\renewcommand{\labelitemi}{$\blacksquare$}
\item \textbf{Scalability}: Distributed algorithms decompose large, complex problems into smaller subproblems. This enables fast computations and makes the approach scalable.
\item \textbf{Privacy and Sovereignty Preservation}: With distributed optimization, data privacy is better preserved because detailed information, such as grid configurations or customer usage behaviors: does not need to be shared. Each entity can maintain the confidentiality of its own data, which is crucial in collaborations where different entities own and operate separate parts of the system.
\item \textbf{Robustness}: Distributed approaches enhance system reliability and resiliency by mitigating the risk of single-point cyber failures that centralized systems are prone to.
\item \textbf{Adaptability}: These algorithms adapt more quickly to network topology and infrastructure changes without requiring a complete system overview. This enables flexible reconfiguration of the system when new components are added or existing ones are modified.
\end{itemize}
Comprehensive surveys on distributed optimization can be found in~\cite{molzahn2017surveyDistr,yang2019survey,patari2021distributed,al2023distributed}.

To solve nonconvex AC \acrshort{opf} problems in a distributed manner, existing methodologies often exploit specific network architectures. For instance, several approaches~\cite{zheng2015fully,peng2016distributed,rajaei2021decentralized} are restricted to fully radial networks (Type V in Table~\ref{TB::Formulation}), which facilitates the use of convex relaxations based on the \acrfull{bfm}. Other studies focus on integrated transmission-distribution (ITD) systems configured in a star-like topology (Types III and IV). These methods typically employ a master-worker splitting framework~\cite{sun2014masterslave}, where worker nodes solve decoupled local subproblems while a central master iteratively coordinates boundary variables. Because the master problem itself does not require convexity, this framework can accommodate meshed subnetworks via the \acrfull{bim}~\cite{LinChenhui2020,wang2022nested}.

\begin{table}[htp!]
\caption{Types of distributed AC OPF Problems} \label{TB::Formulation}\vspace{-6pt}
\centering
\footnotesize
\sisetup{separate-uncertainty, table-format=1.4(5), detect-all}
\setlength{\tabcolsep}{5pt}
\renewcommand{\arraystretch}{1.1}
\begin{tabular}{ccccc}\toprule
Type & Network & Topology & Model & Problem Type \\\toprule
$\RN{1}$ &Partitioned & Generic & BIM & Nonconvex  \\
$\RN{2}$ &Merged & Generic & BIM & Nonconvex \\
$\RN{3}$ &Merged & Stellate & BIM & Nonconvex\\
$\RN{4}$ &Merged & Stellate & Hybrid & Partially Convexified\\
$\RN{5}$ &Merged & Stellate & BFM & Convexified\\
\bottomrule
\end{tabular}
\end{table}
\begin{table}[htp!]
\centering
\footnotesize
\sisetup{separate-uncertainty, table-format=1.4(5), detect-all}
\caption{Distributed optimization for solving AC OPF} \label{TB::Distri_Opt}\vspace{-6pt}
\noindent    
\setlength{\tabcolsep}{1pt} 
\def\arraystretch{0.8}
\begin{tabular}{ccccccccccc}\toprule
\multirow{2}{*}{Ref.} & \multirow{2}{*}{Type} & \multirow{2}{*}{Algorithm} & Global/Local &  \multirow{2}{*}{$N^\text{bus}$}  &    \multirow{2}{*}{Execution}  &  \multirow{2}{*}{Speed} &  \multirow{2}{*}{Accuracy} \\
&  & & Convergence&&&\\\toprule
\cite{zheng2015fully,peng2016distributed,rajaei2021decentralized} & $\RN{5}$ & ADMM & - &  $10^3$  & Sequential & + & +\\
\cite{LinChenhui2020,wang2022nested} & $\RN{4}$ & DCC &  Global &$10^3$ & Sequential & +++ &+++\\  
\cite{tu2020two} & $\RN{3}$ & Distr. IPM &  Local&$10^7$ & Parallel & +++ &\\
\cite{engelmann2021essentially} & $\RN{2}$ & Distr. IPM &  Local &$10^2$ & Sequential & +++ & +++\\
\cite{lu2017fully} & $\RN{2}$ & Distr. IPM &  - &$10^4$ & Sequential & ++ & +++\\
\cite{dai2023itd} & $\RN{2}$ & ALADIN & Global &$10^3$ & Sequential & +++ &+++\\\hline    
\cite{6748974} & $\RN{1}$ & ADMM & -   &$10^2$  & Sequential & + & +\\
\cite{Guo2017}& $\RN{1}$ & ADMM &  - &  $10^3$ & Sequential & + & +\\
\cite{mhanna2018adaptive}  & $\RN{1}$  & ADMM & - &$10^4$  & Sequential & + & + \\
\cite{sun2021two,sun2023two} & $\RN{1}$ & ADMM &  Global &$10^4$ & Parallel & + & +\\\hline
\cite{engelmann2017distributed,Engelmann2019} & $\RN{1}$ & ALADIN &  Global &$10^2$  & Sequential & +++ &+++\\
\cite{dai2023easimov} & $\RN{1}$ & ALADIN & Global &$10^2$ & Distributed & ++ &+++\\
\cite{engelmann2020decomposition} & $\RN{1}$ & ALADIN& Global &$10^1$ & Sequential & +++ &+++\\
This work & $\RN{1}$ & BALADIN & Global &$10^5$ & Parallel & ++++ &+++\\
\bottomrule
\end{tabular}
\end{table}
{
Some distributed primal-dual interior-point methods have been developed for DC OPF~\cite{minot2016parallel} and for convex problems~\cite{tran2013combining,khoshfetrat2017distributed}. More recent work relaxes convexity requirements~\cite{tu2020two,Ali2024DistributedPI}, but these approaches are typically restricted to a star-like architecture (Type~III). In addition, their convergence and runtime can be highly sensitive to the complexity of the master problem. Related studies~\cite{dai2023itd,engelmann2021essentially,lu2017fully} move toward more complex connection settings (Type II), and validate performance on small-to-medium cases with only a few interconnecting lines. This leaves the practical impact of dense coupling across regions that remains insufficiently explored.}

{In contrast, this work targets the full nonconvex AC OPF at large scale and compares distributed algorithms systematically across different decompositions and coupling densities. We partition the network into interconnected subgrids (Type I), which captures realistic meshed coupling and does not rely on a master-worker hierarchy.} By reformulating the problem into a standard distributed form, generic optimization algorithms can be applied across various power system control tasks. Among these, the \acrfull{admm} has been widely adopted. Early implementations were limited to small-scale cases~\cite{6748974}, but subsequent innovations in power-flow-based partitioning~\cite{Guo2017} and adaptive acceleration schemes~\cite{mhanna2018adaptive} have enabled \acrshort{admm} to handle systems exceeding 10,000 buses. Recent two-level frameworks~\cite{sun2021two,sun2023two} even provide global convergence guarantees for cases as large as 30,000 buses. However, as a first-order method, \acrshort{admm} typically requires thousands of iterations to achieve modest accuracy. This high iteration count necessitates an expensive communication infrastructure with high bandwidth and low latency to manage the intensive coordination between agents.

To accelerate convergence, \acrfull{aladin} was introduced in \cite{Boris2016aladin}, merging the strengths of \acrshort{admm} and \acrfull{sqp}. This approach achieves global convergence and local quadratic convergence rates for nonconvex problems when the Hessian is chosen appropriately. When applied to AC \acrshort{opf} in \cite{engelmann2017distributed,Engelmann2019}, it reached high accuracy for the IEEE 300-bus case in only 26 iterations. Despite its promise, ALADIN faces significant synchronization delays in geographically distributed environments due to the transmission of full KKT matrices~\cite{dai2023easimov}. While null-space and Schur-complement techniques have been proposed to reduce communication overhead~\cite{engelmann2020decomposition}, scalability remains a challenge. In particular, although the method scales well when inequality constraints are absent~\cite{Boris2016aladin,muhlpfordt2021distributed,dai2022rapid}, its reliance on active-set methods can hinder scalability because of the combinatorial complexity these methods entail~\cite[Ch.~15.2]{nocedal2006numerical}. This limitation of standard ALADIN has also been highlighted in recent studies~\cite{dai2023itd,lanza2024ADMM_ALADIN,Dai2025thesis}.

Despite the theoretical advantages, critical issues remain unsolved before the full potential of distributed optimization can be realized. This paper investigates generic distributed optimization for real-world, large-scale nonconvex problems, specifically targeting the challenges identified by~\cite{patari2021distributed}:
\begin{itemize}[leftmargin=15pt]
\renewcommand{\labelitemi}{$\blacksquare$}

\item \textbf{Convergence Speed}: Many current methods require thousands of iterations to reach acceptable accuracy, necessitating faster algorithms capable of real-time grid adaptability.

\item \textbf{Convergence Guarantees}:  Rigorous theoretical frameworks are essential to ensure that distributed solutions consistently converge within practical timeframes.

\item \textbf{Scalability}: Algorithms must demonstrate robust performance on massive power system benchmarks under diverse and realistic operational conditions.

\item \textbf{Communication Efficiency}: Data exchange must be minimized to remain compatible with existing communication technologies and infrastructure.

\end{itemize}

The contributions of the paper are three-fold:
\begin{enumerate}[leftmargin=15pt]
\item We propose a two-level distributed algorithm for generic \acrshort{nlp}s with convergence guarantees and fast local convergence speed. The upper level uses barrier methods for inequality constraints, while the lower level applies \acrshort{aladin} to smoothed equality-constrained subproblems, which avoids combinatorial challenges~\cite[Ch.~15.2]{nocedal2006numerical}. Integration of the Schur complement for derivative condensation reduces computational \& communication overhead and mitigates data privacy risks from sparsity patterns in Jacobian and Hessian matrices.

\item We demonstrate enhanced scalability through extensive large-scale simulations with different operational scenarios. Using the \acrfull{spmd} paradigm with distributed memory, the approach outperforms the centralized solver \ipopt~\cite{wachter2006ipopt} on modest hardware, accounting for synchronization overhead.  
\ipopt~\cite{wachter2006ipopt} on modest hardware.

\item We systematically analyze how network decomposition impacts convergence, communication overhead, and practical performance of the proposed distributed algorithm. Additionally, we outline criteria for optimal decomposition of the network and efficient distributed \acrfull{ad} enabled by the proposed distributed algorithm.

\end{enumerate}

The rest of this paper is organized as follows: Section~\ref{sec::problem} describes the AC ~\acrshort{opf} problem and its distributed formulation. Section~\ref{sec::algorithm} introduces the proposed algorithm with convergence analysis. Section~\ref{sec::distIP::implementation} provides practical considerations, including network decomposition, distributed \acrshort{ad}, and resilience to single-point failures. Section~\ref{sec::distIP::test} discusses the large-scale simulation results in detail, while Section~\ref{sec::communication} examines communication efforts under both theoretical and practical scenarios. Finally, Section~\ref{sec::conclusion} concludes the paper.

\section{Problem Formulation}
\label{sec::problem}
This section introduces the standard AC \acrfull{opf} formulation. We then explore graph-based decomposition methods and demonstrate how these can be reformulated into a generic distributed approach with affine consensus constraints.
\subsection{Preliminaries}
Let us consider a power system $\mathcal{S}=(\mathcal{N},\;\mathcal{L},\;\mathcal{R})$, where $\mathcal{N}$ represents the set of buses, $\mathcal{L}$ the set of lines and $\mathcal{R}$ the set of regions into which the network is partitioned. Let $n^\mathcal{N}$, $n^\mathcal{L}$, and $n^\mathcal{R}$ denote the number of buses, lines, and regions, respectively. In this context, the complex voltage at a bus can be expressed in rectangular coordinates, i.e.,
$V_i = u_i+ \textbf{j} w_i,$ 
where $u_i$ and $w_i$ are the real and imaginary parts of the complex voltage $V_i$, respectively, and $\textbf{j}= \sqrt{-1}$ represents the imaginary. The variables $p^g_i$, $q^g_i$ (resp. $p^l_i$, $q^l_i$) represent the real and reactive power produced by generators (resp. loads) at bus $i$. These variables are set to 0 if no generator (resp. load) is connected to a bus $i$. The real and imaginary parts of the complex nodal admittance matrix $Y$ are represented by $G$ and $B$. The optimization state vector $x$ includes all the real and imaginary parts of complex voltage and active and reactive generator injections, i.e., $x = (u,w,p^g,q^g)\in \mathbb{R}^{4 n^\mathcal{N}}$. 

\subsection{Conventional AC Optimal Power Flow}
The AC \acrfull{opf} problem with the complex voltage in rectangular coordinates can be written as follows
\begin{subequations}\label{eq::opf::rec}
\begin{flalign}
\min_{x} & \; \sum_{i \in \mathcal{N}}\left\{a_{i,2} \left(p^{g}_i\right)^2+a_{i,1}\;p^{g}_{i} + a_{i,0}\right\}\\
\textrm{s.t.}& \;  p_i^\text{inj}(w,u) = p^g_i-p^l_i,& \forall i\in\mathcal{N}\label{eq::bim::balance::p}\\
& \; q_i^\text{inj}(w,u) = q^g_i-q^l_i ,&\forall i\in\mathcal{N}\label{eq::bim::balance::q}\\
& \; p^2_{ij}+q^2_{ij}\leq \overline{s}^2_{ij},&\forall(i,j)\in\mathcal{L}\label{eq::bim::limit::line}\\
& \;{\tan \underline{\theta}_{ij}\leq \frac{u_j w_i-u_i w_j}{u_i u_j+w_i w_j}\leq\tan\overline{\theta}_{ij},}
&\forall(i,j)\in\mathcal{L}\label{eq::bim::limit::angle}\\
& \;\underline{v}_i^2\leq u_i^2+w_i^2\leq \overline{v}^2_i,&\forall i\in\mathcal{N}\label{eq::bim::limit::voltage}\\
& \; \underline{p}_i^g\leq p_i^g\leq \overline{p}_i^g,&\forall i\in\mathcal{N}\label{eq::bim::limit::pg}\\
& \; \underline{q}_i^g\leq q_i^g\leq \overline{q}_i^g,&\forall i\in\mathcal{N}\label{eq::bim::limit::qg}
\end{flalign}
\end{subequations}
with power injection at bus $i\in\mathcal{N}$,
\begin{subequations}
\begin{align}
p_{i}^\text{inj}(u,w)=&\sum_{k\in\mathcal{N}}{G_{ik} (u_i u_k + w_i w_k)+ B_{ik}(w_i u_k - u_i w_k)}, \notag\\
q_{i}^\text{inj}(u,w)=&\sum_{k\in\mathcal{N}}{ G_{ik} (w_i u_k - u_i w_k)- B_{ik}(u_i u_k + w_i w_k)},\notag
\end{align}
\end{subequations}
and power flow in line $(i,j)\in\mathcal{L}$,
\begin{subequations}
\begin{align}
p_{ij}=&G_{ij} (u_i^2-u_i u_j+w_i^2-w_i w_j) + B_{ij}(u_i w_j-w_i u_j), \notag\\
q_{ij}=&B_{ij}\{u_i(u_j-u_i)+w_i(w_j-w_i)\} + G_{ij}(u_i w_j - w_i u_j),\notag
\end{align}
\end{subequations}
where $a_{i,2}$, $a_{i,1}$, and $a_{i,0}$ are the polynomial coefficients for the operation cost of power generations at bus $i$. The symbols $\underline{\cdot}$ and $\overline{\cdot}$ denote upper and lower bounds for the respective state variables. Thus, problem~\eqref{eq::opf::rec} includes the active and reactive nodal power balances~\eqref{eq::bim::balance::p}-\eqref{eq::bim::balance::q}, apparent power limit {and angle difference limit} on transmission lines~\eqref{eq::bim::limit::line}- \eqref{eq::bim::limit::angle} and the bounds on the voltage magnitudes and power generations~\eqref{eq::bim::limit::voltage}-\eqref{eq::bim::limit::qg}.

\begin{remark}[\textbf{Decomposition of OPF~\cite{shin2021graph}}]
AC \acrshort{opf} problems~\eqref{eq::opf::rec} are well-suited for distributed approaches and parallel computing due to their repetitive nature across different components e.g. buses, lines, generators. 
\end{remark}
\begin{figure}[htbp!]
\centering
\begin{minipage}[b]{0.21\textwidth}
\vspace{0pt}
\begin{subfigure}{\linewidth}
\centering
\includegraphics[width=0.8\linewidth]{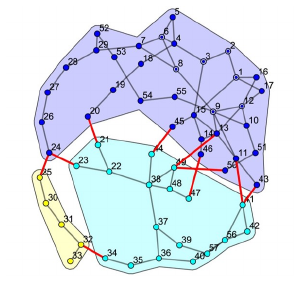}\vspace{-5pt}
\caption{Graph decomposition~\cite{murray2020grid}}
\label{fig::graph::decomposition}
\end{subfigure}

\medskip

\begin{subfigure}{\linewidth}
\centering
\includegraphics[width=0.9\linewidth]{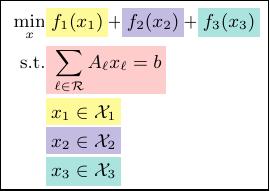}\vspace{-2pt}
\caption{Graph-based formulation}
\label{fig::graph::reformulation}
\end{subfigure}
\end{minipage}\hfill
\begin{subfigure}[b]{0.275\textwidth}
\vspace{0pt}
\centering
\includegraphics[width=\linewidth]{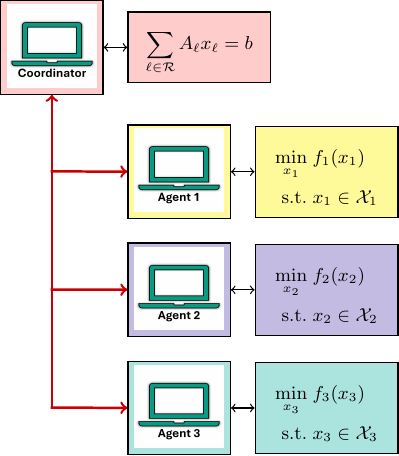}\vspace{-3pt}
\caption{Distributed optimization framework}
\label{fig::distOPT}
\end{subfigure}
\caption{Graph-based distributed optimization in power systems}
\label{fig::graphOPT}
\end{figure}

\subsection{Graph-Based Reformulation for Distributed Optimization}\label{sec::opf::distributed}
Regarding the distributed problem formulation, we share components with neighboring regions to ensure physical consistency, following~\cite{muhlpfordt2021distributed}. For a given region $\ell\in\mathcal{R}$, buses entirely within the region form the core bus set  $\mathcal{N}_\ell^\textrm{core}$, while those shared by neighboring regions are copy buses in $\mathcal{N}_\ell^\textrm{copy}$. Hence, the full bus set is in region $\ell$ is given by $\mathcal{N}_\ell = \mathcal{N}_\ell^{\textrm{core}} \cup \mathcal{N}_\ell^{\textrm{copy}}$. Additionally, $\mathcal{L}_\ell$ represents the set of all transmission lines within region $\ell$. This spatial decomposition, visualized in Fig.~\ref{fig::graphOPT} using a partitioned IEEE 57-bus system,
creates local subproblems while introducing shared variables at regional boundaries that require coordination.

The distributed reformulation process, illustrated in Fig.~\ref{fig::graph::reformulation}, can transform a conventional AC OPF problem~\eqref{eq::opf::rec} into a distributed framework where: 1) each region $\ell$ solves a local subproblem with its objective $f_\ell$ and feasible set $\mathcal{X}_\ell$, and 2) a consensus mechanism enforces physical consistency across regions.
This yields the distributed formulation:
\begin{subequations}
\label{eq::problem::original}
\begin{align}
\min_{x}\quad&f(x):=\sum_{\ell\in\mathcal{R}} f_\ell(x_\ell)\\\label{eq:affine1}
\textrm{s.t.}  \quad  &\sum_{\ell\in\mathcal{R}} A_\ell x_\ell =b\\\label{ineq::nonlinear}
&c^\textrm{E}_\ell(x_\ell) = 0\qquad\quad\mid\ell\in\mathcal{R}\\\label{eq::nonlinear}
&c^\textrm{I}_\ell(x_\ell) \leq 0\qquad\quad\mid\ell\in\mathcal{R}
\end{align}
\end{subequations}
where local state $x_\ell$ includes {the real and imaginary part of the complex voltage $u_i, w_i$} for all buses $i\in\mathcal{N}_\ell$, and the generator injections $p^g_i, q^g_i$ for all core buses $i\in\mathcal{N}^\textrm{core}_\ell$. In a specific region $\ell\in\mathcal{R}$, $f_\ell$ represents the local cost function concerning core generators in the region $\ell$. Function $c^\text{E}_\ell$ encompasses the nodal power balance~\eqref{eq::bim::balance::p}-\eqref{eq::bim::balance::q} for all core buses $i\in\mathcal{N}^\textrm{core}_\ell$, and $c^\text{I}_\ell$ collects system limits on voltage magnitude~\eqref{eq::bim::limit::voltage}, active and reactive power generation~\eqref{eq::bim::limit::pg}-\eqref{eq::bim::limit::qg} for all core bus $i\in\mathcal{N}^\textrm{core}_\ell$, as well as power limit~\eqref{eq::bim::limit::line} on the transmission lines $(i,j)\in\mathcal{L}_\ell$ within region $\ell$. The consensus constraint~\eqref{eq:affine1} ensures consistency of core and copy variables between neighboring regions. 

As shown in Fig.~\ref{fig::distOPT}, this architecture enables parallel computation through regional agents that solve local subproblems, coordinated by a central mechanism enforcing consensus. This approach preserves physical consistency while improving scalability, as regional solutions only require coordination at shared boundaries rather than full system visibility.

\section{Distributed Nonconvex Optimization}
\label{sec::algorithm}
This section develops our distributed method for large-scale, smooth, nonlinear, nonconvex problems with local constraints and affine coupling. We first summarize the limits of standard ALADIN and then motivate the design choices in Algorithm~\ref{alg::distIP}.
\subsection{Limitations of Standard ALADIN}
\acrfull{aladin}, first proposed in~\cite{Boris2016aladin}, was developed for generic noncovex \acrfull{nlp} with guarantees for global convergence. Unlike the first-order \acrshort{admm}, \acrshort{aladin} integrates a \acrfull{sqp} framework. This involves using the active-set method and solving the equality-constrained \acrfull{qp} problems in the coordinator. This enables quadratic convergence rates and strong scalability for problems without inequality constraints~\cite{Boris2016aladin,dai2022rapid}.

A key limitation arises when inequality constraints are present. Similar to traditional \acrshort{sqp} approaches, the active-set method introduces combinatorial challenges~\cite[Ch.~15.2]{nocedal2006numerical}, as the number of possible active-sets increases exponentially with the number of inequalities. This scalability challenge of the standard \acrshort{aladin} has been acknowledged in recent studies~\cite{dai2023itd,lanza2024ADMM_ALADIN,Dai2025thesis}. It restricts prior studies on distributed AC \acrshort{opf} using \acrshort{aladin} to systems with fewer than 1000 buses or a merged network with loosely inter-regional coupling, making the problem relatively easier to solve.

Another limitation is communication overhead. The standard ALADIN requires derivative information from each subproblem to construct a second-order approximated model of the original problem~\eqref {eq::problem::original} in the coordinator. While this accelerates convergence compared with first-order algorithms like \acrshort{admm}, it significantly increases computational and communication demands. As system size grows, evaluating the second-order derivatives and coordinating communication between local agents and a central coordinator become significant bottlenecks. Previous research has focused mainly on the numerical performance of the \acrshort{aladin} algorithm, with limited exploration of distributed implementations. This leaves the impact of communication overhead on practical performance unclear. Additionally, optimizing over networks, such as power systems, may reveal the network topology of the power networks through the sparsity pattern of local derivative information, raising privacy concerns.


\begin{algorithm}[htbp!]
\caption{Barrier ALADIN for solving~\eqref{eq::DistProblem::barrier}}
\label{alg::distIP}
\small
\textbf{Input:} 
\begin{itemize}
\item initial primal and dual points $(z,\;\lambda)$,
\item positive penalty parameters\;$\rho$,\;$\mu$,
\item scaling symmetric matrices $\Sigma_\ell\succ 0$
\end{itemize} 
\textbf{Repeat:}
\begin{algorithmic}[1]
\State Solve decoupled \acrshort{nlp}s for all $\ell\in\mathcal{R}$ \Comment{\textit{\scriptsize Parallel}}\label{alg::distIP::decoupled}
\begin{subequations}\label{eq::problem::decoupled::barrier}
\begin{align}
\min_{x_\ell,s_\ell}\quad& f^\mu_\ell(x_\ell,s_\ell) + \lambda^\top A_\ell x_\ell + \frac{\rho}{2} \norm{x_\ell-z_\ell}_{\Sigma_\ell}^2\label{eq::problem::decoupled::barrier::obj}\\
\textrm{s.t.}  \quad  &c^\textrm{E}_\ell(x_\ell) = 0\qquad\quad\mid\gamma_\ell,\\
&c^\textrm{I}_\ell(x_\ell) + s_\ell =  0\,\,\quad\mid\kappa_\ell,
\end{align}
\end{subequations}
with $\kappa_\ell, s_\ell\geq 0$, and use initial guess for $\kappa_\ell, s_\ell$  from~\eqref{alg::distIP::line} if avaliable. \vspace{3pt}
\State Evaluate $E^\mu_\ell$ and $E^0_\ell$ from~\eqref{eq::optimality::local}, and condense derivatives for all $\ell\in\mathcal{R}$\Comment{\textit{\scriptsize Parallel}}\label{alg::distIP::condense}
\begin{subequations}\label{eq::condensing}
\begin{align}
W_\ell =& \; - \bar{A}_\ell \; \bar{H}_\ell^{-1} \bar{A}^\top_\ell\label{eq::condensing::dualHessian}\\
h_\ell=& \; A_\ell x_\ell - \bar{A}_\ell \; \bar{H}_\ell^{-1} \bar{g}_\ell
\end{align}
\end{subequations}
where $\bar{H}_\ell$, $\bar{g}_\ell$ and $\bar{A}_\ell$ are local curvature information at decoupled solution $(x_\ell,\,s_\ell)$; more detailed description see~\eqref{eq::dist::kkt::pd::system}-\eqref{eq::local::curverture}.\vspace{3pt}
\State Send $E^\mu_\ell$, $E^0_\ell$, $A_\ell x_\ell$, $W_\ell$ and $h_\ell$ to coordinator, and 
run Algorithm~\ref{alg::distributed::correction} to correct \text{inertia} if the \text{inertia} condition~\eqref{eq::dist::inertia} is not satisfied\vspace{3pt}\Comment{\textit{\scriptsize Synchronize}}\vspace{3pt}
\State Terminate if condition~\eqref{eq::optimality::original} is satisfied. \vspace{3pt} \Comment{\textit{\scriptsize Centralized}}
\State Update barrier parameter by~\eqref{eq::barrier::update} if condition~\eqref{eq::optimality::barrier} is satisfied. \vspace{3pt}\label{alg::distIP::mu} \Comment{\textit{\scriptsize Centralized}}
\State Solve coordination problem\Comment{\textit{\scriptsize Centralized}}\label{alg::distIP::coordination}
\begin{equation}\label{eq::coordination}
W \Delta \lambda  = - h
\end{equation}
with $W = \sum_{\ell\in\mathcal{R}} W_\ell$ 
\State Send $\Delta \lambda$ back to each local agent \Comment{\textit{\scriptsize Synchronize}}\vspace{3pt}
\State Recover local primal-dual step for all $\ell\in\mathcal{R}$\Comment{\textit{\scriptsize Parallel}}
\begin{subequations}\label{eq::recover}
\begin{align}
\begin{pmatrix}
\Delta x_\ell\\\Delta \gamma_\ell
\end{pmatrix} &= -\; \bar{H}_\ell^{-1} \bar{A}_\ell^\top \Delta \lambda  - \bar{g}_\ell,\label{eq::recovery::dual}\\
\Delta s_\ell &= -\; c^\textrm{I}_\ell(x_\ell) - s_\ell - R_\ell \; \Delta x_\ell,\\
\Delta \kappa_\ell &= -\; \kappa_\ell + \mathcal{S}_\ell^{-1} (\mu\; \mathds{1} - \mathcal{K}_\ell \Delta s_\ell),
\end{align}
\end{subequations}
and obtain local primal-dual steplength $(\beta^\text{p}_\ell,\,\beta^\text{d}_\ell)$ by using the fraction-to-boundary method
\begin{subequations}\label{eq::fraction2boundary}
\begin{align}
\beta^\textrm{p}_\ell =&  \max \{\beta\in(0,1]:\; s_\ell+\beta \Delta s \geq (1-\tau) s_\ell\}\\
\beta^\textrm{d}_\ell =& \max \{\beta\in(0,1]:\; \kappa_\ell+\beta \Delta \kappa \geq (1-\tau) \kappa_\ell\}.
\end{align}     
\end{subequations}
\State Update primal-dual variables by \Comment{\textit{\scriptsize Synchronize}} 
\begin{subequations}\label{alg::distIP::line}
\begin{align}
z^{+}_\ell &= z_\ell +\alpha_1  (x_\ell-z_\ell) + \alpha_2 \, \beta^\textrm{p}\Delta x_\ell, \quad&\forall\ell\in\mathcal{R}&\\
s^+_\ell&=s_\ell+  \alpha_2 \,\beta^\textrm{p} \Delta s_\ell, \quad&\forall\ell\in\mathcal{R}&\\
\kappa^+_\ell&=\kappa_\ell+  \alpha_2 \,\beta^\textrm{p} \Delta \kappa_\ell, \quad&\forall\ell\in\mathcal{R}&\\
\lambda^{+} &= \lambda +  \alpha_3 \, \beta^\textrm{d}\Delta\lambda,
\end{align}
with
\begin{align}\label{step::distIP::coordinate}
\beta^\textrm{p} = \min_{\ell\in\mathcal{R}} \beta^\textrm{p}_\ell,
\end{align}
\end{subequations}    
where, $\alpha_1$, $\alpha_2$, $\alpha_3$ are determined by ~\cite[Alg.~3]{Boris2016aladin}. Alternatively, for full-step updates without globalization, set $$\alpha_1=\alpha_2=\alpha_3=1.$$   
\end{algorithmic}

\end{algorithm}

\subsection{Algorithm Framework}

To overcome these limitations, we solve Problem~\eqref{eq::problem::original} with a two-level framework. 
At the outer level, we replace active-set handling of inequalities with a log-barrier formulation, so each outer iteration solves a smooth equality-constrained subproblem. 
At the inner level, we apply an ALADIN-type coordination scheme, but reduce coordination overhead by Schur complement, communicating only condensed curvature and residual quantities.

By using the barrier to smooth the inequalities, the original problem~\eqref{eq::problem::original} is transformed into a series of equality-constrained barrier problems:
\begin{subequations}\label{eq::DistProblem::barrier}
\begin{align}
\min_{x,s}\quad& f^\mu(x,s) = \sum_{\ell\in\mathcal{R}} f^\mu_\ell(x_\ell,s_\ell)\label{eq::Dist::barrier::obj}\\\label{eq::barrier::consensus}
\textrm{s.t.}  \quad  &\sum_{\ell\in\mathcal{R}} A_\ell x_\ell =b\quad\;\;\mid\lambda\\\label{ineq::nonlinear::barrier}
&c^\textrm{E}_\ell(x_\ell) = 0\qquad\quad\mid\gamma_\ell,\;\ell\in\mathcal{R}\\
&c^\textrm{I}_\ell(x_\ell) + s_\ell =  0\,\,\quad\mid\kappa_\ell,\;\ell\in\mathcal{R}\label{eq::barrier::end}
\end{align}
\end{subequations}
where the local barrier objective is:
\begin{equation}\label{eq::barrierFunc}
  f^\mu_\ell (x_\ell,s_\ell)= f_\ell(x_\ell) - \mu\sum_{m \in\mathcal{C}_\ell^\textrm{I}}\textrm{ln} (s_\ell^{(m)}),
\end{equation}
and $s_\ell\geq0,\; \forall\ell\in\mathcal{R}$ is implicitly enforced because minimization with the barrier term prevents $s_\ell$ from approaching zero~\cite{nocedal2006numerical}. The original local Lagrangian functions are written:
$$
\mathcal{L}_\ell(x_\ell,\,\kappa_\ell,\,\gamma_\ell) = f_\ell(x_\ell) + \kappa^\top_\ell \,c^\mathrm{I}_\ell(x_\ell)+\gamma_\ell^\top\, c^\mathrm{E}(x_\ell),\;\forall\ell\in\mathcal{R}.
$$

The corresponding \acrshort{kkt} conditions of the barrier problems~\eqref{eq::DistProblem::barrier}, the so-called \textit{perturbed \acrshort{kkt} conditions}, can be written as
\begin{subequations}\label{eq::dist::perturbedKKT}
\begin{align}
\nabla_x \mathcal{L}(x_\ell,\kappa_\ell,\gamma_\ell) + A_\ell^\top \lambda =&\;0,\quad\forall \ell\in\mathcal{R}\label{eq::dist::perturbedKKT::1}\\
-\mu \mathds{1}+\mathcal{S}_\ell\;\kappa_\ell=&\;0,\quad\forall \ell\in\mathcal{R}\label{eq::dist::perturbedKKT::slack}\\
c_\ell^\textrm{E}(x_\ell)=&\;0,\quad\forall \ell\in\mathcal{R}\label{eq::dist::perturbedKKT::2}\\
c_\ell^\textrm{I}(x_\ell)+s_\ell=&\;0,\quad\forall \ell\in\mathcal{R}\label{eq::dist::perturbedKKT::3}\\
\sum_{\ell\in\mathcal{R}} A_\ell \; x_\ell - b =&\;0
\end{align}
\end{subequations}
with
$$\nabla_x \mathcal{L}(x_\ell,\kappa_\ell,\gamma_\ell) = \nabla f_\ell(x_\ell) + R_\ell^\top \kappa_\ell + J_\ell^\top \gamma_\ell,$$ 
where $R_\ell = \nabla c^\text{I}(x_\ell)^\top,\;J_\ell = \nabla c^\text{E}(x_\ell)^\top, \;\mathcal{S}_\ell = \textrm{diag}(s_\ell)$, and $\mathds{1}$ denotes the vector of all ones with the respective dimension. 
Note that the conditions~\eqref{eq::dist::perturbedKKT} for $\mu=0$, together with $\kappa_\ell,\;s_\ell\geq0,\;\forall\ell\in\mathcal{R}$, are equivalent to the \acrshort{kkt} conditions for the original problem~\eqref{eq::problem::original}~\cite{wachter2006ipopt}.
\begin{remark} 
The solution to the problem~\eqref{eq::DistProblem::barrier} is given by
\begin{equation}
x^\star(\mu) = \arg\min_{x}~\eqref{eq::Dist::barrier::obj},\;\textrm{s.t.}\;\eqref{eq::barrier::consensus}-\eqref{eq::barrier::end}.\notag
\end{equation}
As the barrier parameter $\mu$ approaches zero, the solution $x^\star(\mu)$ converges to the solution of the original problem~\eqref{eq::problem::original}.
\end{remark}
Then, in the lower level, \acrshort{aladin} is used to solve a smoothed problem~\eqref{eq::DistProblem::barrier} for a fixed value of 
barrier parameter $\mu$ in a distributed manner. Once sufficient accuracy is achieved, the barrier parameter is decreased, and the process continues until the algorithm converges.


\subsubsection{Decoupled Step}
The proposed distributed algorithm is outlined in Algorithm~\ref{alg::distIP} in detail. Firstly, decoupled barrier \acrshort{nlp} subproblems~\eqref{eq::problem::decoupled::barrier} are solved locally for each region $\ell\in\mathcal{R}$. After solving these local subproblems, the derivatives are evaluated and condensed locally at the local solution $(x_\ell,\;s_\ell)$ using the Schur complement. These derivatives are later utilized in the next stages of the distributed optimization process to facilitate communication and coordination between regions. It is important to note that both Step~\ref{alg::distIP::decoupled} and Step~\ref{alg::distIP::condense} can be executed in parallel, enhancing the efficiency of the overall process.

\begin{remark}\label{rmk::f2b}
When solving the decoupled \acrshort{nlp} subproblems, the fraction-to-boundary method~\cite[Ch.~19.2]{nocedal2006numerical} is essential to ensure that the dual variables $\kappa_\ell$ and the slack variables $s_\ell$ remain positive throughout the iteration. This is a critical requirement for the proper functioning of barrier-based approaches.
\end{remark}

\subsubsection{Assessing Optimality Conditions}
To assess the convergence of the inner \acrshort{aladin} algorithm, we need to evaluate the perturbed \acrshort{kkt} conditions~\eqref{eq::dist::perturbedKKT}. After solving the decoupled barrier \acrshort{nlp} subproblems~\eqref{eq::problem::decoupled::barrier}, we first evaluate the decoupled optimality residual~\eqref{eq::dist::perturbedKKT::1}-\eqref{eq::dist::perturbedKKT::3} locally, defined by:
\begin{align}\label{eq::optimality::local}
E^\mu_{\ell}(x_\ell,s_\ell,\lambda&,\kappa_\ell,\gamma_\ell) =\max \left\{\frac{\norm{\nabla_{x} \mathcal{L}_\ell(x_\ell,\lambda,\kappa_\ell,\gamma_\ell)}_\infty}{s^d_\ell},\right.\notag\\&\left.\frac{\norm{S_\ell\;\kappa_\ell-\mu \mathds{1}}_\infty}{s^c_\ell},\norm{ \begin{bmatrix}
c^\text{E}_\ell(x_\ell)\\c^{I}_\ell(x_\ell)+s_\ell
\end{bmatrix}}_\infty\right\}
\end{align}with scaling parameters $s_\ell^d,\;s^c_\ell\geq1$.

Each region $\ell\in\mathcal{R}$ sends its local residual  $E^\mu_{\ell}$ and the coupling term $A_\ell\ x_\ell$ to the coordinator, which computes the global optimality residual as:
\begin{align}\label{eq::residual::global}
E^\mu(x,s,\lambda,\kappa,\gamma) =\max\left\{ \max_{\ell\in\mathcal{R}} \left\{E^\mu_{\ell}(x,s,\lambda,\kappa,\gamma)\right\},\right.\notag\\\left. \norm{\sum_{\ell\in\mathcal{R}} A_\ell x_\ell-b}_\infty\;\right\}.
\end{align}

For the original problem~\eqref{eq::problem::original}, setting $\mu=0$ gives the residual $E^0$. When the current primal-dual iterate $(x,s,\lambda,\kappa,\gamma)$ satisfy:
\begin{equation}\label{eq::optimality::original}
E^0(x,s,\lambda,\kappa,\gamma)\leq \epsilon,
\end{equation}
where $\epsilon$ is a user-defined tolerance, and the slack-dual variables $(s_\ell,\kappa_\ell)$ remain positive at the decoupled solutions (see Remark~\ref{rmk::f2b}), the solution also satisfies the \acrshort{kkt} conditions of the original problem~\eqref{eq::problem::original}. Therefore, the algorithm terminates when condition~\eqref{eq::optimality::original} is met.

Let $\nu$ denote the outer barrier iteration counter and $\mu_\nu$ the current barrier parameter.
At each outer iteration, we solve the barrier subproblem~\eqref{eq::DistProblem::barrier} only to a
$\mu_\nu$-scaled accuracy:
\begin{equation}\label{eq::optimality::barrier}
E^{\mu_\nu}(x,s,\lambda,\kappa,\gamma)\leq \eta^-\ \mu_\nu,
\end{equation}
for a constant $\eta^->0$. In our implementation, the tolerance parameter is set as $\eta^-=10$. Once \eqref{eq::optimality::barrier} is satisfied, we
reduce the barrier parameter according to
\begin{equation}\label{eq::barrier::update}
\mu_{\nu+1} = \max\left\{\frac{\epsilon}{10},\min\left\{\frac{\mu_\nu}{5},\mu_\nu^{1.5}\right\}\right\}
\end{equation}
Defining $\mu_{\min}:=\epsilon/10$, this update reaches
$\mu_{\min}$ after finite number of outer updates, unless
\eqref{eq::optimality::original} triggers termination earlier.
\begin{remark}\label{rmk:BLN_strategy2}
To promote fast local convergence, we use the barrier update strategy in~\cite[Strategy~2]{byrd1997local}\cite{wachter2006ipopt}, which is proven to give rise to superlinear convergence under standard second-order sufficient conditions.
\end{remark}

\subsubsection{Condensed Coordination Step}

The coordination problem~\eqref{eq::coordination} is one iteration of the second-order multiplier method~\cite{bertsekas1997nonlinear}. Consider applying a Newton step to the primal-dual perturbed \acrshort{kkt} conditions~\eqref{eq::dist::perturbedKKT}:
\hspace{-10pt}\begin{equation}\label{eq::dist::kkt::pd::system}
\begin{bmatrix}
\nabla_{xx}^2 \mathcal{L} & 0 & J^\top & R^\top& A^\top\\
0 & \mathcal{K} & 0 & \mathcal{S} &0 \\
J & 0& 0& 0&0\\
R & I& 0& 0&0\\
A &0 &0 &0 &0
\end{bmatrix}\hspace{-5pt}
\begin{bmatrix}
\Delta x\\
\Delta s\\
\Delta \gamma\\
\Delta \kappa \\
\Delta \lambda
\end{bmatrix}\hspace{-2pt}= \hspace{-2pt} -\hspace{-2pt}\begin{bmatrix}
g\\ \mathcal{K}\; s-\mu \mathds{1} \\c^\text{E}(x)\\c^\text{I}(x)+s\\  Ax-b
\end{bmatrix}\hspace{-3pt},
\end{equation}
where the \acrshort{kkt} matrix is asymmetric, to obtain this linear system. By eliminating the slack variables $\Delta s$ and dual variables $\Delta \kappa$, we can obtain a symmetric linear system. This results in the following coupled system:
\begin{align}\label{eq::coupled::barrier::newton}
\begin{bmatrix}
H & J^\top & A^\top\\
J& &\\
A
\end{bmatrix}\begin{pmatrix}
\Delta x\\\Delta\gamma\\\Delta \lambda
\end{pmatrix} = -\begin{pmatrix}
g \\
c^\textrm{E}(x)\\
Ax-b
\end{pmatrix}
\end{align}
with
\begin{align*}
H =&\; \textrm{blkdiag}(H_\ell),\;H_\ell \approx \nabla_{xx}\mathcal{L}_\ell +R_\ell^\top \mathcal{S}_\ell^{-1}\mathcal{K}_\ell R_\ell,\\
g =&\; \textrm{vertcat}(g_\ell),\quad g_\ell = \nabla_x \mathcal{L}_\ell + R_\ell^\top  \mathcal{S}_\ell^{-1} \left(\mu \; \mathds{1} + \mathcal{K}_\ell c_\ell^\textrm{I}(x_\ell)\right),\\
J =&\; \text{blkdiag}(J_\ell),\;\; A = \text{horzcat}(A_\ell).
\end{align*}
The equivalent \acrshort{qp} subproblem
\begin{subequations}\label{eq::distIP::subQP}
\begin{align}
\min_{\Delta x}& \sum_{\ell\in\mathcal{R}}\left\{\frac{1}{2} \Delta x_\ell^\top\;H_\ell \;\Delta x_\ell + g_\ell^\top \Delta x_\ell\right\}\\
\textrm{s.t.}  & \quad  \sum_{\ell\in\mathcal{R}} A_\ell \Delta x_\ell = b- A\, x  \quad \mid\lambda+\Delta \lambda\\
&  \quad J_\ell \Delta x_\ell = -c_\ell^\text{E}(x_\ell), \;\;\ell \in \mathcal{R}\;.
\end{align}
\end{subequations}
is similar to the coupled QP step in the standard \acrshort{aladin}~\cite[Alg.~2]{Boris2016aladin}.

Next, we follow the Second-Order Multiplier method~\cite{bertsekas1978convergence} to further condense~\eqref{eq::coupled::barrier::newton} using the Schur complement. 
Reordering the elements in the Newton step~\eqref{eq::coupled::barrier::newton}, it can be rewritten as
\begin{align}\label{eq::blkdiag::kkt}
\begin{bmatrix}
\bar{H}_1  & & & \bar{A}_1^\top\\
& \ddots & & \vdots\\
& & \bar{H}_n & \bar{A}_n^\top\\ 
\bar{A}_1 & \cdots & \bar{A}_n & 0
\end{bmatrix}\;\begin{pmatrix}
\Delta \bar{x}_1\\ \vdots \\ \Delta \bar{x}_n \\ \Delta \lambda
\end{pmatrix} = - \begin{pmatrix}
\bar{g}_1\\
\vdots\\
\bar{g}_n\\
Ax-b
\end{pmatrix}
\end{align}
with local primal-dual iterates $\Delta \bar{x}_\ell =\left(\Delta x_\ell^\top,\,\Delta \gamma_\ell^\top \right)^\top$ and curvature information
\begin{align}\label{eq::local::curverture}
\bar{H}_\ell = \begin{bmatrix}
H_\ell & J^\top_\ell\\
J_\ell &  
\end{bmatrix},\; \bar{g}_\ell = \begin{pmatrix}
g_\ell\\ c_\ell^\text{E} (x)
\end{pmatrix}\,  \text{and}\; \bar{A}_\ell = \begin{bmatrix}
A_\ell & 0
\end{bmatrix}
\end{align}
or written in a compact form as an alternative:
\begin{align}\label{eq::blkdiag::kkt::22}
\begin{bmatrix}
\bar{H}  & \bar{A}^\top\\
\bar{A}  & 0
\end{bmatrix}\;\begin{pmatrix}
\Delta \bar{x} \\ \Delta \lambda
\end{pmatrix} = - \begin{pmatrix}
\bar{g}\\
Ax-b
\end{pmatrix}
\end{align}
with block diagonal matrix $\bar{H} = \text{blkdiag}(\bar{H}_\ell)$. By using the Schur complement, we have
\begin{equation*}\label{eq::kkt::som}
W \Delta \lambda  = - h
\end{equation*}
with
\begin{align*}
W = \sum_{\ell\in\mathcal{R}} W_\ell,\;h = - b + \sum_{\ell\in\mathcal{R}} h_\ell,
\end{align*}
where $W_\ell$ and $h_\ell$ can be computed in each agent $\ell$ in parallel
\begin{align*}
W_\ell =& \; - \bar{A}_\ell \; \bar{H}_\ell^{-1} \bar{A}^\top_\ell\\
h_\ell=& \; A_\ell x_\ell - \bar{A}_\ell \; \bar{H}_\ell^{-1} \bar{g}_\ell
\end{align*}
Once the coordinator receives the dual Hessian $W_\ell$ and the dual gradient $h_\ell$ from all local agents $\ell\in\mathcal{R}$, and a descent direction can be guaranteed, then we solve the coordination problem in dual space~\eqref{eq::coordination}; otherwise, Algorithm 2 is called to correct the \text{inertia} of the dual Hessian $W$, which is a symmetric matrix. Details about the criterion of descent direction and corresponding modification will be discussed in the following section.     
\begin{remark}\label{rmk::schur::criterions}
The Schur complement method is particularly effective when 1) $\bar{H}_\ell$ is easy to invert or 2) the number of consensus constraints, denoted by $N^\lambda$, is small, which reduces the complexity of the matrix-matrix product $\bar{H}^{-1}_\ell \bar{A}_\ell^\top$. 
\end{remark}
\begin{remark}    
Further condensing the KKT linear systems~\eqref{eq::coupled::barrier::newton} serves to reduce the communication effort between agents and the coordinator, such that only the dual Hessian $W_\ell$ and the dual gradient $g_\ell$ need to be communicated. 
This not only reduces the data exchanged between agents and the coordinator but also distributes part of the coordinator's computational tasks to the local agents. By allowing these tasks to be handled in parallel by the agents, the overall computational burden on the coordinator is reduced, thereby improving the total computation time.
\end{remark}

Once the coordination problem~\eqref{eq::coordination} is solved, the dual step $\Delta \lambda$ is sent back to the local agents, and the full primal-dual step can be recovered locally by~\eqref{eq::recover}:
\begin{align*}
\begin{pmatrix}
\Delta x_\ell\\\Delta \gamma_\ell
\end{pmatrix} &= -\; \bar{H}_\ell^{-1} \bar{A}_\ell^\top \Delta \lambda  - \bar{g}_\ell,\\
\Delta s_\ell &= -\; c^\textrm{I}_\ell(x_\ell) - s_\ell - R_\ell \; \Delta x_\ell,\\
\Delta \kappa_\ell &= -\; \kappa_\ell + \mathcal{S}_\ell^{-1} (\mu\;\mathds{1} - \mathcal{K}_\ell \Delta s_\ell),
\end{align*}

Having recovered local variables in each agent, primal and dual steplengths $\beta^\textrm{p}_\ell$, $\beta^\textrm{d}_\ell~\in~(0,1]$  are determined by the fraction-to-the-boundary rule in each local agent $\ell\in\mathcal{R}$, with
\begin{equation}\label{eq::tau}
\tau = \max{(\tau_{\min},\,1-\mu)}
\end{equation}
By evaluating steplengths locally, no full step information is required during the coordination.

The primal-dual iterates are then updated according to~\eqref{alg::distIP::line}, where the shortest primal step length~\eqref{step::distIP::coordinate} is used. If a globalization strategy is required, the parameters $\alpha_1$, $\alpha_2$, $\alpha_3$ are set as defined by ~\cite[Alg.~3]{Boris2016aladin}. Alternatively, for full-step updates, set $\alpha_1=\alpha_2=\alpha_3=1$.

Additionally, the dual variables $(\gamma_\ell,\kappa_\ell)$ can be updated locally by
\begin{align*}
\kappa^+_\ell&=\kappa_\ell+ \alpha_2 \, \beta^\textrm{d}_\ell\Delta \kappa_\ell, \quad&\forall\ell\in\mathcal{R}&\\
\gamma^+_\ell&=\gamma_\ell+ \alpha_2 \,\beta^\textrm{p}  \Delta\gamma_\ell, \quad&\forall\ell\in\mathcal{R}
\end{align*}
These updates help refine the initial guess for the decoupled \acrshort{nlp}s~\eqref{eq::problem::decoupled::barrier} in subsequent iterations with additional communication overhead.


\subsection{Distributed Inertia Correction}\label{sec::distributedIC}
In nonconvex optimization, a standard Newton step may not provide a descent direction if the KKT system is not regular. We define the inertia of a symmetric matrix to assess this regularity.
\begin{definition}[\textbf{Inertia~\cite{gould1985practical}}]
The inertia of a symmetric matrix $K$ is the triple: $$\mathrm{inertia}(K) = (n^+,n^-,n^0),$$ where $n^+$, $n^-$, and $n^0$ denote the number of positive, negative, and zero eigenvalues of $K$, respectively.
\end{definition}
To ensure that the quadratic programming (QP) subproblems derived from our nonconvex AC \acrshort{opf} are well-behaved, the KKT matrix must satisfy specific inertia conditions:
\begin{theorem}[\textbf{Inertia Condition~\cite[Thm.~{2.1}]{gould1985practical}}]\label{thm::inertia_condition}
Let $K$ be a KKT matrix defined as
\[
K = \begin{bmatrix} H & J^\top \\ J & 0 \end{bmatrix}, 
\quad \text{with } H \in \mathbb{R}^{n^x \times n^x}, \, J \in \mathbb{R}^{n^c \times n^x}.
\]
Suppose $J$ is of full row rank and 
\begin{equation}\label{eq::inertia}
\mathrm{inertia}(K) = (n^x, \, n^c, \, 0).
\end{equation}
Then, $H$ is positive definite on the null space of $J$, and the corresponding QP subproblem has a strict minimizer.
\end{theorem}

In a centralized setting, the perturbed Newton step~\eqref{eq::coupled::barrier::newton} provides a strict descent direction if the global KKT matrix
\begin{equation}\label{eq::kkt::distributed}
K = \begin{bmatrix}
H & J^\top & A^\top\\
J& &\\
A
\end{bmatrix}
\end{equation}
with $K\in\mathbb{R}^{(N^x+N^E+N^\lambda)\times (N^x+N^E+N^\lambda)}$ satisfies:
\begin{equation}\label{eq::perturbedKKT::condition}
\mathrm{inertia}(K) = (N^x ,\, N^\text{E}+N^\lambda, \, 0)
\end{equation}
However, in the proposed distributed framework (Algorithm~\ref{alg::distIP}), the central coordinator does not have direct access to the global KKT matrix~\eqref{eq::kkt::distributed}. 

To bridge this gap, we utilize the Haynsworth Inertia Additivity property to the related inertia of the global KKT matrix~\eqref{eq::kkt::distributed} to the local data handled by individual agents:
\begin{lemma}[\textbf{Haynsworth Inertia Additivity~\cite{haynsworth1968determination}}]\label{lem::additivity}
Let $K$ be a Hermitian matrix partitioned as
\[
K = \begin{bmatrix} K_{11} & K_{21}^\top \\ K_{21} & K_{22} \end{bmatrix}.
\]
If the submatrix $K_{11}$ is nonsingular and $\Lambda = K_{22} - K_{21} K_{11}^{-1} K_{21}^\top$ is the Schur complement of $K_{11}$ in $K$, then
\[
\mathrm{inertia}(K) = \mathrm{inertia}(K_{11}) + \mathrm{inertia}(\Lambda).
\]
\end{lemma}

Applying this lemma to the proposed condensed system enables verification of the global descent condition via parallel local evaluations.
\begin{corollary}[\textbf{Distributed Inertia Condition}]\label{cor::descent}
Let $K$ be the global KKT matrix~\eqref{eq::kkt::distributed}
and let the global Jacobian $\bar{J} = [J^\top, A^\top]^\top$ be of full row rank. Suppose the dual Hessian $W$ satisfies:
\begin{equation}\label{eq::dist::inertia}
\mathrm{inertia}(W) = (N^x, \, N^\text{E}+N^\lambda, \, 0) - \sum_{\ell \in \mathcal{R}} \mathrm{inertia}(\bar{H}_\ell),
\end{equation}
then $H$ is positive definite on the null space of the Jacobian $\bar{J}$ and the equivalent QP subproblem has a strict local minimizer.
\end{corollary}

\begin{proof}
As a a direct result of Lemma~\ref{lem::additivity}, we have 
\begin{align*}
\mathrm{inertia}(K) =& \;\mathrm{inertia}(W)+  \mathrm{inertia}(\bar{H}) \\
    &\,\Downarrow \text{block diagonal structure of $\bar{H}$ in~\eqref{eq::blkdiag::kkt}}\notag\\
  =& \;\mathrm{inertia}(W) + \sum_{\ell\in\mathcal{R}} \mathrm{inertia}(\bar{H}_\ell).
\end{align*}
Then, the conclusion follows from condition~\eqref{eq::perturbedKKT::condition}.
\end{proof}

Condition~\eqref{eq::dist::inertia} decentralizes the verification of the global inertia requirement~\eqref{eq::perturbedKKT::condition}. By evaluating the local inertia of $\bar{H}_\ell$ in parallel, agents provide the necessary data for the coordinator to verify the descent properties of the condensed system. Inspired by the centralized IPOPT framework, we propose a Distributed Inertia Correction (Algorithm~\ref{alg::distributed::correction}) to dynamically regularize the system whenever condition~\eqref{eq::dist::inertia} is violated.
{
\begin{algorithm}
\caption{Distributed inertia Correction}\label{alg::distributed::correction}
\small
\textbf{Input:} $\delta^\text{last}_x$, $\mathrm{inertia}(W)$, $\mathrm{inertia}(\bar{H}),\;\forall \ell\in\mathcal{R}$
\begin{algorithmic}[1]
\If{the \text{inertia} condition~\eqref{eq::dist::inertia} is satisfied}
\State $\delta^\text{last}_x = 0$
\Else
\If{$\delta^\text{last}_x= 0$}
\State $\eta^\text{inc} = \eta^\text{fast}$ and $\delta^x = \delta^{x}_0$
\Else
\State  $\eta^\text{inc} = \eta^\text{slow}$ and $\delta^x = \max\left(\eta^\text{red}\;\delta^x_0,\;\bar{\delta}^x\right) $
\EndIf \State\textbf{end if}
\If{ $n^0 = 0$}
\State $\delta^\gamma=0$
\Else
\State   $\delta^\gamma=\delta^\gamma_0$
\EndIf \State\textbf{end if}
\While{the \text{inertia} condition~\eqref{eq::dist::inertia} is not satisfied}
\State $ \mathrm{inertia}(\bar{H}_\ell)= \mathrm{inertia}\left(\begin{bmatrix}
    H_\ell + \delta^x I & J_\ell^\top\\
    J_\ell &  -\delta^\gamma I
\end{bmatrix}\right)$
\State $ \mathrm{inertia}(W)= \mathrm{inertia} \left(-\sum_{\ell\in\mathcal{R}} \bar{A}_\ell \; \bar{H}_\ell^{-1} \bar{A}^\top_\ell\right)$
\State $\delta^\text{last}_x = \delta^x$
\State $\delta^x = \eta^\text{inc} \delta^x$
\EndWhile \State\textbf{end while}
\EndIf \State\textbf{end if}
\end{algorithmic}
\textbf{Return} $\delta^\text{last}=\delta^x$
\end{algorithm}
}

Algorithm~\ref{alg::distributed::correction} is triggered whenever the distributed inertia condition~\eqref{eq::dist::inertia} fails. It initializes the primal regularization parameter $\delta^x$ from the previous Newton step and then increases $\delta^x$ gradually until the inertia condition~\eqref{eq::dist::inertia} is satisfied. This ensures the smallest perturbation needed to recover a well-behaved search direction, while preserving the distributed structure of the method.

\begin{remark}\label{rem::inertia::correction}
The dual regularization $\delta^\gamma>0$ addresses rank deficiency in the constraint Jacobian (LICQ failure), while the primal regularization $\delta^x>0$ enforces a curvature condition consistent with second-order sufficiency (SOSC). Together, they allow progress even when these regularity conditions do not hold at intermediate iterations.
\end{remark}
\begin{remark}
Although the evaluation of the inertia of the local matrix $\bar{H}_\ell$ can be evaluated in parallel by local agents during the condensation of the derivatives~\eqref{eq::condensing}, the condensation process itself is not computationally inexpensive. It is particularly expensive if the local agents are densely coupled and  $N^\lambda$ is relatively large, as noted in Remark~\ref{rmk::schur::criterions}. 
Furthermore, Algorithm~\ref{alg::distributed::correction} introduces additional iterative communication, which can further increase the overall computational overhead.
\end{remark}
\subsection{Global Convergence Guarantees}\label{sec::global::convergence}
{This section establishes global convergence guarantees for Algorithm~\ref{alg::distIP} applied to smooth, nonconvex problems in the distributed form~\eqref{eq::problem::original}. We begin by stating the basic
smoothness assumption used in both the global and local convergence analyses:
\begin{assumption}[\bf{Smoothness}]\label{ass::regular::distrIP}
For each region $\ell\in\mathcal{R}$, the local functions $f_\ell$, $c_\ell^\mathrm{E}$ and $c_\ell^\mathrm{I}$ are twice Lipschitz-continuously differentiable on the local feasible set $\mathcal{F}_\ell := \{x_\ell \mid c_\ell^\mathrm{E}(x_\ell)=0,\; c_\ell^\mathrm{I}(x_\ell)\le 0\}$.
\end{assumption}

\changes{Since the barrier subproblem~\eqref{eq::DistProblem::barrier}
with fixed $\mu>0$ is a smooth equality-constrained NLP, we
adapt the globalization analysis in
\cite[Sec.~6]{Boris2016aladin} to the fixed-$\mu$ barrier
problem. The following construction is the direct counterpart
of the proximal primal problem~\cite[Eqs.~(6.5)]{Boris2016aladin}, and its partial dual problem
\cite[Eqs.~(6.3)]{Boris2016aladin}, with the original
objective replaced by the barrier objective~\eqref{eq::barrierFunc}. For notational simplicity, the dependence on the fixed parameter $\mu$ is suppressed.}

\changes{For a reference point $\bar{x}$, define the proximal primal
problem
\begin{subequations}\label{eq::Zrho}
\begin{align}
Z_\rho(\bar{x})
:=
\min_{x,s}\quad&
\sum_{\ell\in\mathcal{R}}
\left\{
f_\ell^\mu(x_\ell,s_\ell)
+
\frac{\rho}{2}
\|x_\ell-\bar{x}_\ell\|_{\Sigma_\ell}^2
\right\}
\\
\mathrm{s.t.}\quad&
\sum_{\ell\in\mathcal{R}}A_\ell x_\ell=b,
\label{eq::Zrho::consensus}
\\
&
c_\ell^\mathrm{E}(x_\ell)=0,
\qquad \ell\in\mathcal{R},
\\
&
c_\ell^\mathrm{I}(x_\ell)+s_\ell=0,
\qquad \ell\in\mathcal{R}.
\end{align}
\end{subequations}
The condition $s_\ell>0$ is enforced by the domain of the
logarithmic barrier in $f_\ell^\mu$.
Dualizing only the affine consensus constraint
\eqref{eq::Zrho::consensus} gives the decoupled dual
function
\begin{subequations}\label{eq::Vrho}
\begin{align}
V_\rho(\bar{x},\lambda)
:=
\min_{x,s}\quad&
\sum_{\ell\in\mathcal{R}}
\Big\{
f_\ell^\mu(x_\ell,s_\ell)
+
\lambda^\top A_\ell x_\ell\notag\\
&\qquad\quad+
\frac{\rho}{2}
\|x_\ell-\bar{x}_\ell\|_{\Sigma_\ell}^2
\Big\}
-\lambda^\top b
\\
\mathrm{s.t.}\quad&
c_\ell^\mathrm{E}(x_\ell)=0,
\qquad \ell\in\mathcal{R},
\\
&
c_\ell^\mathrm{I}(x_\ell)+s_\ell=0,
\qquad \ell\in\mathcal{R}.
\end{align}
\end{subequations}
\begin{assumption}[\textbf{Regularity for Globalization}]
\label{ass::globalization}
For every reference point $\bar{x}$ at which
Step~\ref{alg::global::dual} of
Algorithm~\ref{alg::basicGlobal} is invoked, the partial
Lagrangian dual of~\eqref{eq::Zrho} attains its maximum with
zero duality gap, i.e.,
\[
\operatorname{sup}_{\lambda}V_\rho(\bar{x},\lambda)
=
Z_\rho(\bar{x}).
\]
\end{assumption}
Assumption~\ref{ass::globalization} states directly the
auxiliary primal-dual property used by step~\ref{alg::global::dual} of
the globalization routine~\ref{alg::basicGlobal}. In
\cite[Lem.~1]{Boris2016aladin}, this property is established
for sufficiently large $\rho$ by assuming regularity of the
minimizer of an additional auxiliary projection problem~\cite[Eq.~(6.7)]{Boris2016aladin}.
Houska et al. further note that this projection regularity
condition is rather strong and that related no duality gap
results can be established under weaker regularity conditions;
see the detailed discussion inß
\cite[Sec.~6]{Boris2016aladin}. We impose the required
primal-dual property directly to avoid introducing that
additional projection problem. Hence, the following result is an adaptation of
the globalization result in
\cite[Thm.~2]{Boris2016aladin} to interior point methods:
\begin{theorem}[\textbf{Global Convergence}]
\label{Thm}
Let Assumptions~\ref{ass::regular::distrIP} and
\ref{ass::globalization} hold. For every fixed $\mu>0$,
suppose that the barrier problem
\eqref{eq::DistProblem::barrier} is feasible and bounded from
below. Let the matrices
$\Sigma_\ell$ be positive definite and the penalty parameter
$\rho$ be sufficiently large. If the line-search
parameters $\alpha_1$, $\alpha_2$, and $\alpha_3$ are
determined by Algorithm~\ref{alg::basicGlobal}, then
Algorithm~\ref{alg::distIP} terminates after a finite number
of iterations.
\end{theorem}}
\begin{proof}
\changes{The proof adapts the two-part argument of
\cite[Thm.~2]{Boris2016aladin}. First, fix $\mu>0$.
Problems~\eqref{eq::Zrho} and~\eqref{eq::Vrho} are the
fixed-$\mu$ barrier counterparts of
\cite[Eqs.~(6.5) and~(6.3)]{Boris2016aladin}.
Assumption~\ref{ass::globalization} supplies directly the
dual-attainment and zero-duality-gap property established by
\cite[Lem.~1]{Boris2016aladin}. Therefore, under the
remaining conditions of Theorem~\ref{Thm}, the globalization
argument in~\cite[Thm.~2]{Boris2016aladin} carries over to
the fixed-$\mu$ inner iteration. In particular, the
dual-only step of Algorithm~\ref{alg::basicGlobal} cannot be
executed infinitely often, while the accepted primal steps
yield a strict decrease of the merit function. Consequently,
the fixed-$\mu$ inner iteration reaches the prescribed
positive stopping tolerance after finitely many iterations.}

\changes{Second, whenever~\eqref{eq::optimality::barrier} holds, the barrier parameter is updated according to~\eqref{eq::barrier::update}. This update reaches $\mu_{\min}=\epsilon/10$ after a finite number of outer iterations. At a KKT point of the barrier problem with $\mu=\mu_{\min}$, condition~\eqref{eq::dist::perturbedKKT::slack} gives $\mathcal{S}_\ell\kappa_\ell=\mu_{\min}\mathds{1}$, while the remaining components of $E^0$ vanish. Since the scaling factor $s_\ell^c\geq1$ in~\eqref{eq::optimality::local}, we have $E^0\leq\mu_{\min}<\epsilon$. Hence, the termination condition~\eqref{eq::optimality::original} is satisfied after a finite number of total iterations.}
\end{proof}

\subsection{Local Convergence Rate}\label{sec::local::convergence}
{
To characterize the local convergence rate of the inner ALADIN iterations, we impose the following regularity conditions at the decoupled solutions of the local barrier NLPs.
\begin{assumption}[\textbf{Regularity for local Convergence}]\label{ass::hessian::distrIP}
Consider the barrier subproblem~\eqref{eq::DistProblem::barrier} with a fixed $\mu>0$. At the decoupled solution $x_\ell$ of the decoupled barrier NLP~\eqref{eq::problem::decoupled::barrier}, the following hold for each region $\ell\in\mathcal{R}$:
\begin{enumerate}[label=(\alph*)]
\item the Hessian approximations $H_\ell$ satisfy \label{ass::HH}
$$
\norm{\nabla^2_{xx}\mathcal{L}_{\ell}
+ R_\ell^\top \mathcal{S}_\ell^{-1}\mathcal{K}_\ell R_\ell-H_\ell}
\;=\; \mathcal{O}\!\left(\norm{x_\ell - z_\ell}\right),
$$
\item the penalty parameter $\rho$ is sufficiently large such that\label{ass::rho}
$$
\nabla^2_{xx}\mathcal{L}_{\ell}
+ R_\ell^\top \mathcal{S}_\ell^{-1}\mathcal{K}_\ell R_\ell
+ \rho\,\Sigma_\ell
$$
is positive definite on the null space of $J_\ell$.
\end{enumerate}
Additionally, at the solution $x^*$ of the original problem~\eqref{eq::problem::original}, the following conditions are satisfied:
\begin{enumerate}[label=(\alph*)]
\setcounter{enumi}{2}
\item The \acrfull{licq}, the \acrfull{scc}, and the \acrfull{sosc}.\label{ass::reguarity}
\end{enumerate}
\end{assumption}}

{Assumption~\ref{ass::hessian::distrIP}\,\ref{ass::HH} ensures that $H_\ell$ is a sufficiently accurate
second-order approximation near the decoupled solutions $x_\ell$, so the regularization in
Algorithm~\ref{alg::distributed::correction} does not destroy local quadratic convergence.
Assumption~\ref{ass::hessian::distrIP}\,\ref{ass::rho} is mild, since it can always be enforced by taking
$\rho$ large enough (cf.~Theorem~\ref{Thm}). 
If an adaptive adjustment scheme for $\rho$ is needed, one may
use the iterative update in~\cite[Alg.~2, Step~5]{Engelmann2019}, or increase $\rho$ only when the coupling
residual $\norm{Ax-b}$ fails to decrease sufficiently~\cite[Sec.~4.2.2]{bertsekas1997nonlinear}.}

{Assumption~\ref{ass::hessian::distrIP}\,\ref{ass::reguarity} ensures that the solution $x^*$ of the original problem is regular. By the implicit function theorem~\cite[Thm.~2.1]{fiacco1976sensitivity}\cite[Thm.~A.2]{nocedal2006numerical}, there exists a locally unique and stable primal-dual central path $x^\star(\mu)$ for sufficiently small $\mu>0$, with $x^\star(\mu)\to x^\star(0)=x^*$ as $\mu\to 0$. This also facilitates the analysis of local convergence rates below.}

{In practice, LICQ, SCC, and SOSC may be temporarily violated away from this neighborhood. The proposed approach remains well defined by regularizations: the distributed inertia correction (Algorithm~\ref{alg::distributed::correction}) regularizes the local KKT systems whenever the inertia condition~\eqref{eq::dist::inertia} fails, which addresses negative curvature, related to SOSC failure, and rank deficiency in the constraint Jacobian, related to LICQ failure (cf.~Remark~\ref{rem::inertia::correction}). Meanwhile, the fraction-to-boundary rule~\eqref{eq::fraction2boundary} maintains strict positivity of the local slack variables $s_\ell$ and the corresponding dual variables $\kappa_\ell$, supporting SCC. These safeguards are inspired by centralized regularization~\cite{vanderbei1999interior,wachter2006ipopt} and are included to improve robustness in the nonconvex and occasionally degenerate regimes encountered along practical iterates.
}

{
Applying ALADIN to solve the barrier subproblems~\eqref{eq::DistProblem::barrier} with fixed $\mu$ can establish local quadratic convergence of the inner loop of Algorithm~\ref{alg::distIP}.
The following result is a direct application of~\cite[Thm.~1]{ZhaiJunyialadin} to the barrier subproblem~\eqref{eq::DistProblem::barrier} for a fixed barrier parameter $\mu$:
\begin{theorem}   
[\textbf{Quadratic Convergence to $x^\star(\mu)$}]\label{thm::local::aladin::distrIP}
Suppose Assumptions~\ref{ass::regular::distrIP} and~\ref{ass::hessian::distrIP} hold. 
For a fixed $\mu>0$ that is sufficiently small, the inner iterates $x$ generated by Algorithm~\ref{alg::distIP} converge to the barrier solution $x^\star(\mu)$ at a locally quadratic rate, if the full step size is applied, i.e., $\alpha_1= \alpha_2=\alpha_3= 1$.
\end{theorem}
\changes{
Note that introducing the slack variables does not itself destroy
LICQ. In particular, if LICQ holds for the original problem
\eqref{eq::problem::original} at a feasible point $x$, then
\(
[A^\top, J^\top]^\top
\)
has full row rank. The Jacobian of the equality constraints in the barrier problem~\eqref{eq::DistProblem::barrier} with
respect to $(x,s)$ is
\[
\begin{bmatrix}
A & 0\\
J & 0\\
R & I
\end{bmatrix}.
\]
Because of the identity block associated with the slack variables, this matrix has full row rank whenever $[A^\top,J^\top]^\top$ has full row rank. Hence, LICQ of the original problem~\eqref{eq::problem::original} implies LICQ of the barrier reformulation~\eqref{eq::DistProblem::barrier} locally around $x^*$.}

For the outer loop, we employ the parameter update routines in~\cite[Strategy~2]{byrd1997local}\cite{wachter2006ipopt}, which yield local superlinear convergence.
\begin{corollary}[\textbf{Superlinear Convergence of Alg.~\ref{alg::distIP}}]\label{cor::local::alg1}
Suppose the assumptions of Theorem~\ref{thm::local::aladin::distrIP} hold. 
Consider Algorithm~\ref{alg::distIP} with the barrier-parameter update~\eqref{eq::barrier::update}, the $\mu$-scaled inner-loop accuracy requirement~\eqref{eq::optimality::barrier}, and the fraction-to-boundary parameter update~\eqref{eq::tau}. 
Then the iterate $x$ generated by Algorithm~\ref{alg::distIP} converges to the solution $x^*$ at a locally superlinear rate.
\end{corollary}
\begin{proof}
The update rules~\eqref{eq::optimality::barrier}~\eqref{eq::barrier::update}~\eqref{eq::tau} follow the update routine proposed by~\cite[Strategy~2]{byrd1997local}. 
In particular, the barrier update~\eqref{eq::barrier::update} together with the associated $\mu$-scaled inner-loop accuracy requirement~\eqref{eq::optimality::barrier} yields local superlinear convergence for full steps $\beta^{\mathrm{p}}=\beta^{\mathrm{d}}=1$; cf.~\cite[Thm.~2.5]{byrd1997local}. 
Moreover, the update routine for fraction-to-boundary parameter~\eqref{eq::tau} preserves the local convergence property while using fraction-to-boundary rule~\eqref{eq::fraction2boundary} to enfoce $s_\ell>0$ and $\kappa_\ell>0$ for all $\ell\in\mathcal{R}$; cf.~\cite[Thm.~3.2]{byrd1997local}. 
Combining this outer-loop result with the local quadratic convergence of the inner loop to $x^\star(\mu)$ (Theorem~\ref{thm::local::aladin::distrIP}) proves the claim.
\end{proof}}

\section{Practical Considerations for Implementation}\label{sec::distIP::implementation}
In this section, we discuss practical considerations for implementing the proposed Algorithm~\ref{alg::distIP} within a distributed or parallel computing environment. To illustrate the interactions in the distributed process, the process sequence of the proposed Algorithm~\ref{alg::distIP} between the coordinator and local agents is depicted in Fig.~\ref{fig::distIP::SeqFig}. This figure emphasizes the flow of information and synchronization steps, visualizing the coordination for better understanding and practical deployment.
\begin{figure}[htbp!]
\centering
\includegraphics[width=0.9\linewidth]{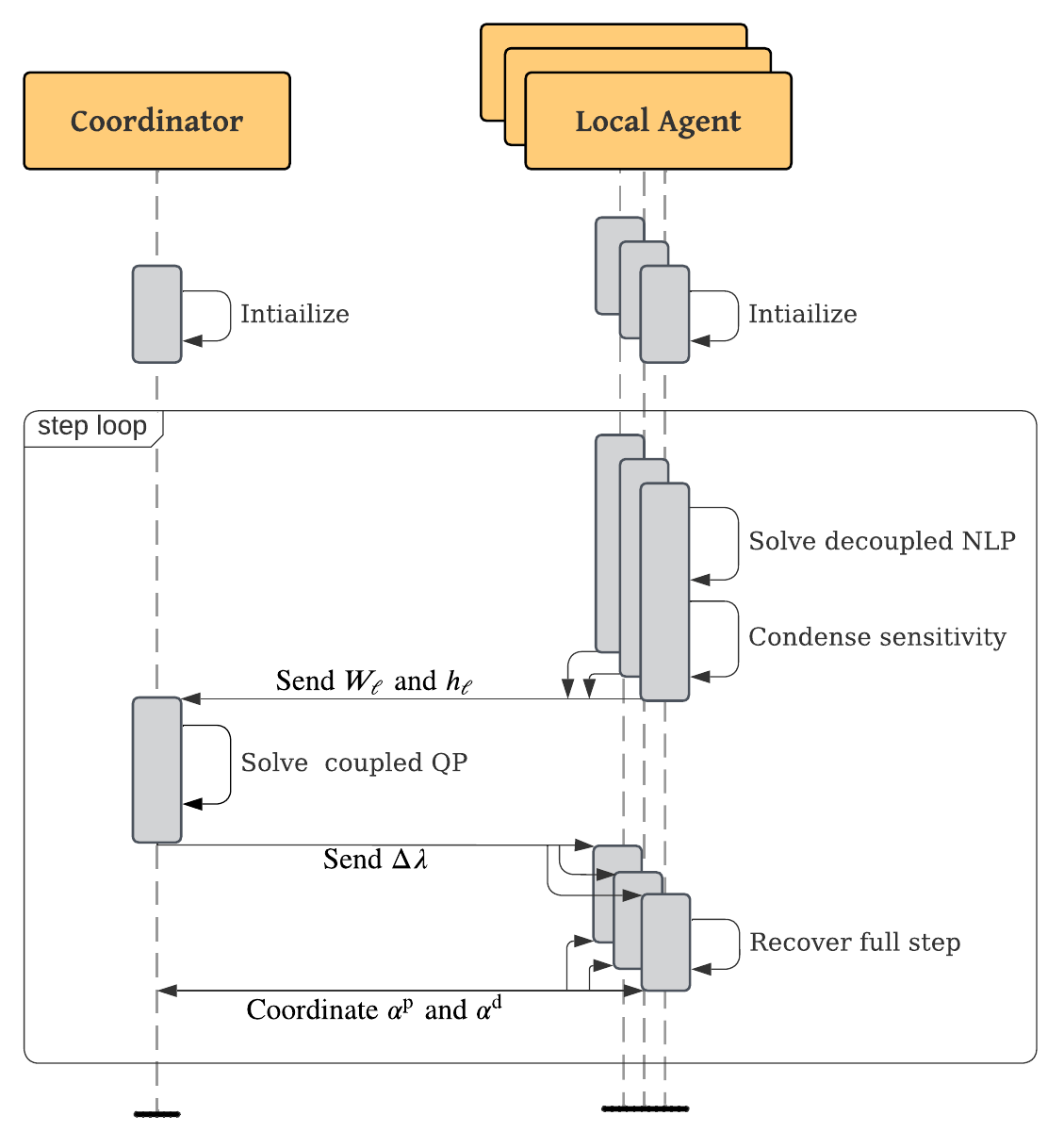}
\caption{Sequence diagram of the proposed Algorithm~\ref{alg::distIP}}
\label{fig::distIP::SeqFig}
\end{figure}

To compare different networks and network decomposition, we define two measures:
\begin{itemize}[leftmargin=15pt]
\item {\bf{Network Density:}} average number of lines per bus,
\begin{equation}\label{eq::network::density}
\zeta = \frac{N^\text{line}}{N^\text{bus}}\in(0,\zeta_{\max}],\quad\mathrm{with}\;
\zeta_{\max}:=\frac{N^{\mathrm{bus}}-1}{2}.
\end{equation}
It measures how densely buses are interconnected and is independent of the partition.

\item {\bf{Coupling Density:}} fraction of coupling variables within a region $\ell\in\mathcal{R}$,
$$\xi_\ell = \frac{N^\text{cpl}_\ell}{N^x_\ell}\in [0,1],$$
where $N^\text{cpl}_\ell$ is the number of coupling variables, and $N^x_\ell$ is the total number of state variables in that region.
The average coupling density across all regions is
$$\overline{\xi} = \frac{1}{N^\text{reg}}\sum_{\ell\in\mathcal{R}}\xi_\ell= \frac{1}{N^\text{reg}}\sum_{\ell\in\mathcal{R}}\frac{N^\text{cpl}_\ell}{N^x_\ell}.$$ 
\end{itemize}
\begin{remark}
For distributed optimization problems~\eqref{eq::problem::original}, $N^{\text{cpl}}_\ell$ equals the number of nonzero rows in $A_\ell$. This indicates how many variables in the region $\ell$ are shared with neighboring regions.
\end{remark}

\subsection{Optimal Graph-based Decomposition}\label{sec::decomposition}
How the power grid is partitioned is essential for reducing problem complexity and enabling efficient parallel computation. Traditional partitioning rules, such as operator-defined areas, hierarchical clustering, or electrical-distance heuristics, are not designed for distributed optimization and often ignore computational coupling, which can strongly affect algorithm performance~\cite{Guo2017}. {In particular, the partition influences distributed methods in two ways: it determines the size of the interfaces between regions, captured by the coupling densities $\xi_\ell$ and $\overline{\xi}$, which drives the per-iteration costs of condensation, coordination, and communication; and it shapes the coupling structure itself, leading to different consensus constraints~\eqref{eq:affine1} and thus different optimal multipliers $\lambda^\star$.}

{Recent work has proposed optimal decomposition strategies beyond operator-defined areas. For example, spectral partitioning has been used to accelerate first-order algorithms such as \acrshort{admm} on large networks~\cite{guo2015intelligent,Guo2017}. However, these approaches are typically tailored to first-order algorithms and often rely on operating-point information to construct the partition.}

{In contrast, Algorithm~\ref{alg::distIP} uses second-order coordination: \acrshort{aladin} solves the
barrier subproblems~\eqref{eq::DistProblem::barrier}, and second-order information is condensed via a
Schur-complement step. We summarize the partitioning effects on globalization and local efficiency in the following remarks:
\begin{remark}[\bf Partitioning and Global Convergence]\label{rmk::global}
Partitioning affects the global behavior through the coupling structure and the associated optimal
multiplier $\lambda^\star$. In particular, it can influence the efficiency of the globalization routine
Algorithm~\ref{alg::basicGlobal}. When the initial multiplier is poor, i.e.,
$\|\lambda^0-\lambda^\star\|$ is large, the dual line-search step~\eqref{eq::alpha3} may accept smaller
steps and lead to slower progress in early iterations.
\end{remark}
In our AC OPF experiments, the initialization is typically close enough that local efficiency dominates convergence performance. We therefore emphasize the local effect of partitioning in the following:
\begin{remark}[\bf Partitioning and Local Convergence]\label{rmk::partition}
The local convergence analysis (Corollary~\ref{cor::local::alg1}) shows that the \emph{local} convergence rate of Algorithm~\ref{alg::distIP}, measured in iteration count, does not depend on the specific graph partition. Hence, varying the number of partitions is not expected to significantly affect the number of iterations. In practice, small differences may still arise for numerical reasons, since finer partitions typically create larger and denser interfaces, which can affect the numerical behavior of the KKT condensing step~\eqref{eq::condensing}; see~\cite{gill1991solving,shin2024accelerating}. By contrast, the primary effect of increasing the number of partitions is on the per-iteration computational cost: derivative condensation~\eqref{eq::condensing}, the coordination solve~\eqref{eq::coordination}, and the associated communication become more expensive, even though the local subproblems~\eqref{eq::problem::decoupled::barrier} themselves become smaller.
\end{remark}
From a practical perspective, local efficiency is mainly governed by two partitioning criteria:
\begin{itemize}[leftmargin=15pt]
\item {\bf Balanced Regions:} decoupled NLPs~\eqref{eq::problem::decoupled::barrier} should be comparable in size and computational effort to reduce idle time under synchronization.
\item {\bf Small Interfaces:} the number of coupling variables should be small, i.e., small $\xi_\ell$ and $\bar{\xi}$, which reduces cost for derivative condensation~\eqref{eq::condensing} and keeps the coordination system~\eqref{eq::coordination} small.
\end{itemize}
As shown in Section~\ref{sec::test::partition}, condensation and coordination can dominate runtime at large scale. Therefore, reducing interface size is often more important than further shrinking the local NLPs.}

{Motivated by these criteria, we adopt \acrshort{kaffpa} from the \acrfull{kahip} package~\cite{sanders2011engineering}. \acrshort{kaffpa} requires only the network graph $(\mathcal{N},\mathcal{L})$ and directly targets balanced partitions with a small cut, i.e., few interconnecting lines, which aligns with the requirements above. Prior ALADIN-based OPF studies~\cite{murray2020grid,murray2021optimal} also report that \acrshort{kaffpa} can outperform spectral partitioning on small- to medium-scale benchmarks. For details on \acrshort{kaffpa}, we refer to~\cite{sanders2011engineering,sandersschulz2013}.}


\subsection{Distributed Automatic Differentiation}

To run optimization algorithms in parallel, some earlier works~\cite{migdalas2013parallel,schnabel1995view} indicate that running an optimization algorithm fully in parallel generally requires parallelizing the function evaluations, including evaluating derivatives in parallel. The challenge is to evaluate the derivatives while minimizing the communication between the different processes. This has led to the development of different prototypes for the Message Passing Interface (MPI) based parallel modelers~\cite{colombo2009structure,watson2012pysp,rodriguez2023scalable}. Other research adopted approaches mainly from machine learning, using fast \acrfull{ad} libraries to efficiently generate derivatives on hardware accelerators such as Graphics Processing Units (GPUs)~\cite{bradbury2018jax,paszke2019pytorch,shin2024accelerating,pacaud2024parallel}.

Unlike state-of-the-art parallel optimization methods, the proposed distributed algorithm for solving NLPs, e.g. AC \acrshort{opf}s, partitions the problem at the network level. Therefore, neither the decoupled \acrshort{nlp}s~\eqref{eq::problem::decoupled::barrier} nor the condensing steps~\eqref{eq::condensing} require global information from neighboring regions, allowing each local agent $\ell \in \mathcal{R}$ to evaluate the objective function $f_\ell$ and constraints $c_\ell$ independently in parallel. This reduces the complexity of communication and coordination using \acrshort{ad}, making the parallel execution of the proposed distributed algorithm more straightforward.

\subsection{Resiliency Against Single-Point Failures}
The proposed distributed algorithm offers resiliency against single-point failures. As shown in the sequence diagram in Fig.~\ref{fig::distIP::SeqFig}, each local agent operates independently during the decoupled step for solving \acrshort{nlp} subproblems and condensing sensitivities. If a single local agent fails, the remaining agents can continue operating, and the centralized coordinator can proceed with a reduced dataset.

Additionally, if the centralized coordinator fails, one local agent can take over the coordination tasks. This would not significantly reduce computation efficiency since local agents are idle during coordination. Moreover, this does not compromise data privacy since both the dual Hessian $W_\ell$ and the dual gradient $g_\ell$ are condensed and data-preserving.

Therefore, the proposed distributed algorithm ensures that failures do not lead to catastrophic breakdowns but instead allow for a degraded yet functional solution. By maintaining a distributed operational framework and relying on local computations, the system mitigates the risk associated with any single point of failure, thereby enhancing overall stability and reliability.

\section{Numerical Test}\label{sec::distIP::test}
This section demonstrates the performance of Algorithm~\ref{alg::distIP} in solving large-scale AC \acrfull{opf} problems under different operating scenarios in a distributed computing environment. We assess the impact of coupling density and validate that this approach outperforms the state-of-the-art centralized nonlinear solver \ipopt~\cite{wachter2006ipopt}.
\subsection{Configuration and Setup}\label{sec::setup}
The framework is built on \matlab-R2023b. The test cases are the largest test cases from the PGLib-\acrshort{opf} benchmark\footnote{The PGLib-\acrshort{opf} is built for benchmarking AC \acrshort{opf} algorithms under IEEE PES Task Force. The benchmark is available on GitHub:~\url{https://github.com/power-grid-lib/pglib-opf} and the baseline results for v23.07 can be found here:~\url{https://github.com/power-grid-lib/pglib-opf/blob/master/BASELINE.md}}~\cite{pglibopf2019power} with version 23.07 and the large-scale test cases from~\cite{kardovs2022beltistos}. The grid is partitioned into a different number of regions by using the \acrshort{kaffpa} algorithm from the \acrshort{kahip} project~\cite{sanders2011engineering}. The case studies are carried out on a small workstation with two \texttt{AMD\textsuperscript{\textregistered} Epyc 7402} 24-core processors, i.e., 48 physical cores in total, and 128 \textsc{GB} installed \textsc{ram}. 

The \casadi toolbox~\cite{andersson2019casadi} is used in \matlab. We use \mab \cite{duff2004ma57} as the sparse linear solver~\footnote{{This choice is based on an empirical screening of HSL solvers on the same workstation and software stack: We tested \maa, \mab, \mac, \mad, and \mae, and \mab yielded the best performance for the benchmark cases considered in this paper.}}. The centralized reference solutions are obtained by solving the AC \acrshort{opf} problems with \ipopt~\cite{wachter2006ipopt}. {All AC \acrshort{opf} instances in this section are formulated and solved in rectangular-power coordinates, as in Problem~\ref{eq::opf::rec}. Accordingly, all reported runtimes, iteration counts, and performance comparisons refer to the rectangular-power formulation.}

To evaluate the performance of the proposed algorithm with minimal communication effort considered, we utilize the \acrfull{spmd} paradigm from the \matlab parallel computing toolbox. This approach facilitates distributed computation by dividing work and data among agents without shared memory. Communication between agents is managed using explicit MPI-style commands, taking into account the potential delays and synchronization requirements. In our setup, one worker functions as the coordinator while the others serve as local agents.
\begin{table*}[htbp!]
\caption{Comparing Different Region Numbers}\label{TB::Nregion}
\centering
\scriptsize
\sisetup{separate-uncertainty, table-format=1.4(5), detect-all}
\setlength{\tabcolsep}{5pt}
\noindent    
\renewcommand{\arraystretch}{1.1}
\begin{tabular}{c cccc ccrrrrrrrrrrrrrrrrrrrrr}\toprule
\multirow{2}{*}{Case}& \multirow{2}{*}{$N^\text{bus}$} & \multirow{2}{*}{$N^\text{gen}$} & \multirow{2}{*}{$N^\text{line}$} & \multirow{2}{*}{$\zeta$} & \multicolumn{5}{l}{\bf Partition Configuration} & \multirow{2}{*}{ Iter}  &\multirow{2}{*}{Init.}&\multicolumn{6}{l}{{\bf Convergence Performance} (Wall Time [s])}  \\\cmidrule(lr){6-10} \cmidrule(lr){13-18}
&&&&& $\abs{\mathcal{R}}$& $\bar{\xi}$ & $N^\text{conn}$ &$N^\lambda$&  $N^x_{\ell,\max}$ &  & & dec. & cond. & coord. & rec. & syn.  & \it total\\\toprule
case78484 & 78.8k & 6.8k & 126.1k &1.607& 10&1.65\% & 395 & 1432 & 18062 & 80&48.40&17.20&188.39&9.27&11.96&4.93&\it231.75\\
&&&&&15&2.45\%&587&2110&12070&78&31.64&9.86&99.90&11.36&7.46&6.06&\it134.64\\
&&&&&20&3.28\%&777&2852&9184&90&24.73&10.65&80.50&22.47&7.75&9.54&\it130.90\\
&&&&&25&4.11\%&990&3600&7530&76&18.70&7.20&81.11&24.88&6.98&11.70&\it131.88\\
&&&&&30&4.17\%&1002&3642&6146&87&15.99&15.54&66.08&23.45&6.17&9.92&\it121.16\\
&&&&&35&5.18\%&1241&4548&5334&76&12.13&6.44&53.23&28.80&5.58&10.75&\it104.80\\
&&&&&40&5.41\%&1299&4756&4734&75&12.70&6.29&37.07&26.20&3.60&10.75&\textbf{\emph{83.90}}\\
&&&&&47&6.37\%&1532&5632&4092&76&11.91&5.73&32.77&32.94&3.35&10.93&\it85.73\\
&&&&&50&6.73\%&1624&5950&3878&79&10.69&8.60&40.06&38.16&4.90&12.25&\it103.97\\
&&&&&55&7.04\%&1690&6232&3614&79&12.12&10.53&59.62&37.92&7.11&13.94&\it129.12\\    \addlinespace    
case193k&192.7k&24.6k&228.6k&1.186&10&0.10\%&55&220&44610&48&49.22&24.05&25.10&0.19&0.42&0.63&\it50.40\\ 
&&&&&15&0.22\%&128&492&29972&46&34.08&13.32&23.60&0.27&0.58&0.58&\it38.34\\ 
&&&&&20&0.32\%&179&702&22514&46&26.72&10.45&23.02&0.56&0.56&0.81&\it35.40\\ 
&&&&&25&0.46\%&257&1004&17986&34&22.14&7.82&14.07&0.37&0.42&0.67&\textbf{\emph{23.35}}\\ 
&&&&&30&0.47\%&272&1050&15072&35&18.72&7.25&15.13&0.83&0.45&8.80&\it32.45\\ 
&&&&&35&0.63\%&351&1368&12830&44&16.85&16.27&14.00&0.68&0.64&2.58&\it34.17\\ 
&&&&&40&0.74\%&414&1612&11424&49&15.28&11.86&21.81&0.89&1.24&6.28&\it42.08\\ 
&&&&&47&0.97\%&549&2138&9656&50&15.66&20.21&34.31&1.54&2.07&10.46&\it68.59\\ 
&&&&&50&1.03\%&575&2252&9086&48&16.05&22.13&30.38&1.72&1.71&1.71&\it57.66\\ 
&&&&&55&1.03\%&585&2278&8350&47&19.01&19.02&21.44&1.81&1.09&6.51&\it49.87\\ 
\toprule
\end{tabular}
\end{table*}

\begin{figure}[htbp!]
\centering
\includegraphics[width=\linewidth]{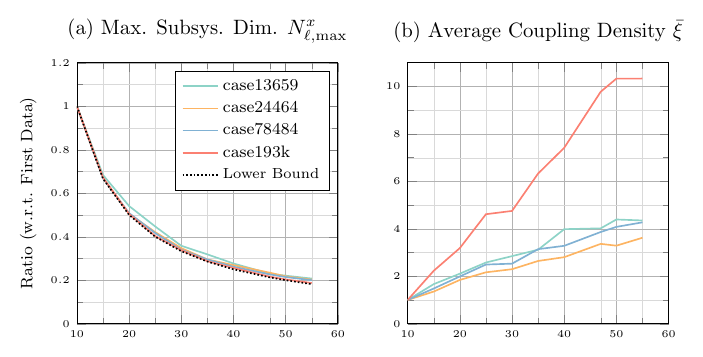}
\caption{Comparison of Network Decomposition on different power systems, i.e., case13659 from~\acrfull{pegase}~\cite{josz2016PEGASE}, case24464 from ARPA-E grid optimization competition~\cite{aravena2023GO}, case78484 from the US Eastern Interconnection states~\cite{snodgrass2021epigrids} and case193k from~\cite{kardovs2022beltistos}.}\label{fig::comparison::network::config} 
\end{figure}

\subsection{Impact of Network Decomposition}\label{sec::test::partition}

We analyze how the network decomposition affects performance by using three large power system test cases from the PGLib-\acrshort{opf} benchmark~\cite{pglibopf2019power}, i.e., case13659, case24464 and case7848, together with one additional large-scale case from~\cite{kardovs2022beltistos}, i.e., case193k. Table~\ref{TB::Nregion} shows detailed results for case78484 and case193k. In general, changes in the number of partitions have a larger effect on overall computing time rather than on the number of iterations required for convergence.

As illustrated in Fig.~\ref{fig::comparison::network::config}, for a given network, when we increase the number of partitions $\abs{\mathcal{R}}$, each local subproblem becomes smaller, but the average coupling density $\bar \xi$ increases. In all four cases, the maximum subproblem size $N^x_{\ell,\max}$ is close to its theoretical lower bound, indicating that the \acrshort{kaffpa} balances the subsystems effectively.
\begin{figure}[htbp!]
\centering
\includegraphics[width=\linewidth]{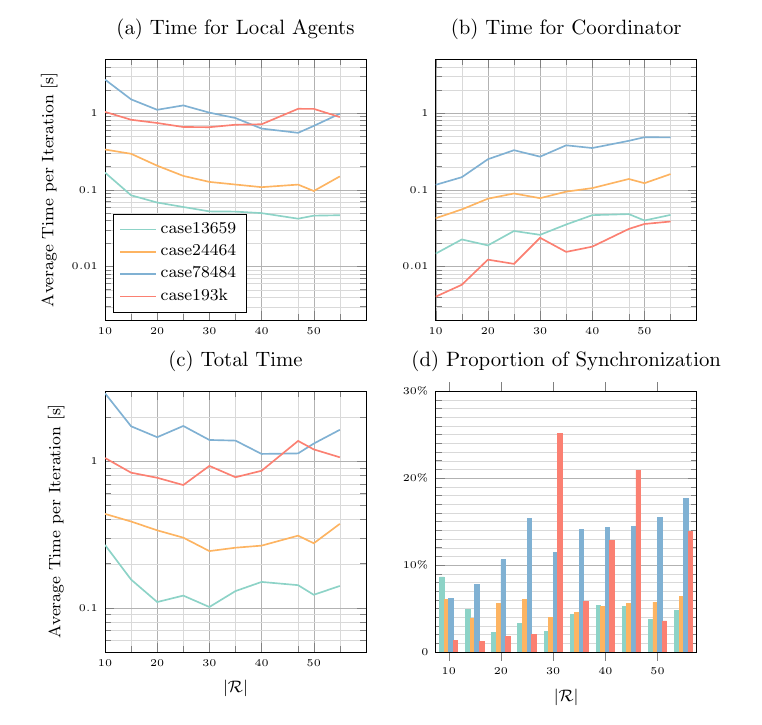}
\caption{Comparison of network decomposition on performance of the proposed Algorithm~\ref{alg::distIP} on large-scale benchmarks.}
\label{fig::comparison::cross::cases} 
\end{figure}

A key difference is how the average coupling density $\bar\xi$ changes. In case193k, $\bar\xi$ grows faster than in the other three networks. This may be due to the network’s topology: unlike the three PGLib-\acrshort{opf} cases with network density $\zeta\approx 1.5$, case193k has a lower network density $(\zeta=1.186)$. When partitioned into $10$ regions, case193k shows an order-of-magnitude lower connectivity than case7848, experiencing faster increases in coupling density $\bar\xi$ than dense networks. 

Fig.~\ref{fig::comparison::cross::cases} and Fig.~\ref{fig::comparison::overall} show how the average wall time per iteration is divided among local agents, the coordinator, and synchronization tasks, including communication overhead. Note that the sudden but simultaneous increases in coordinator (Fig.~\ref{fig::comparison::cross::cases}(b)) and synchronization overhead (Fig.~\ref{fig::comparison::cross::cases}(d)) stem from the additional inertia correction (Algorithm~\ref{alg::distributed::correction}), which dampens the total wall time. Except for that inertia correction, when the partition number is less than or equal to $47$---corresponding to one local agent per core---a higher number of partitions distributes workload to more local agents and shortens their individual computation times.
\begin{figure}[htbp!]
\centering
\includegraphics[width=\linewidth]{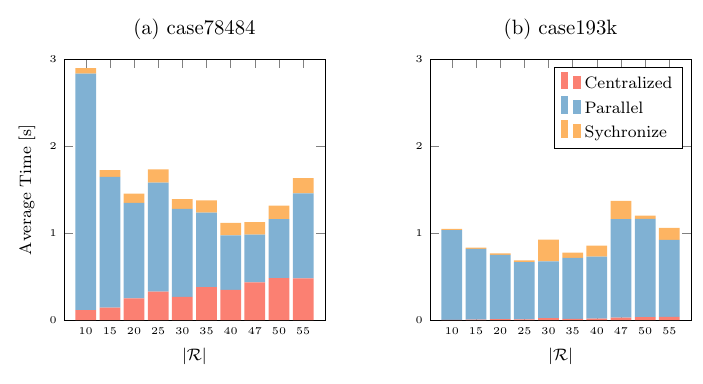}
\caption{Comparison of average wall times per iteration with different partition numbers for case78484 and case193k. Note that the parallel step encompasses the decoupled NLPs, condensing, and recovery steps.}
\label{fig::comparison::overall}
\end{figure}
However, having more partitions also makes it more expensive to condense the derivatives~\eqref{eq::condensing} and enlarge the coordination problems~\eqref{eq::coordination}, leading to higher total wall times.
\begin{remark}
For a given network, increasing the partition number $\abs{\mathcal{R}}$ reduces the size of each subproblem and thus lowers each local agent’s computation effort. However, it also raises the coupling density $\xi_\ell$ in each subproblem, making the local derivative condensing---especially computing $\bar{H}_\ell^{-1}\bar{A}^\top_\ell$ in~\eqref{eq::condensing::dualHessian}---more expensive, and it also increases the size of the coupled \acrshort{qp} subproblems \eqref{eq::coordination} at the coordinator. Hence, when hardware resources are not limited, one must balance these trade-offs to find an optimal number of partitions.
\end{remark}
The trade-off is also observed in Fig.~\ref{fig::comparison::cross::cases}: for the sparser network (case193k), where $25$ partitions yield the shortest total time; beyond that, the total time grows. Although theoretically, partitioning the network can reduce certain computational burdens, real-world performance depends on how local tasks, network density, and coordination overhead interact. More research and testing are needed to determine the best partitioning strategy in practice.





\begin{table*}[htbp!]
\caption{Benchmark}\label{TB::numerical-result}
\centering
\scriptsize
\renewcommand{\arraystretch}{1.1}
\sisetup{separate-uncertainty, table-format=1.4(5), detect-all}
\setlength{\tabcolsep}{3.5pt}
\noindent    
\begin{tabular}{lrrccrrrrrrrrrrrrrrrrrr}
\toprule
\multirow{3}{*}{Case} & \multirow{3}{*}{$N^x$} & \multirow{3}{*}{$N^c$}&\multirow{3}{*}{$\zeta$}& \multirow{3}{*}{$\bar{\xi}$} & \multirow{3}{*}{Mode} &  \multicolumn{8}{l}{{\bf Convergence Performance} (Wall Time [s])} & \multicolumn{4}{l}{\bf Solution Quality}\\
&&&&&&\multicolumn{4}{l}{\baladin}  & \multicolumn{4}{l}{\ipopt} &  \multicolumn{2}{l}{\baladin}  & \multicolumn{2}{l}{\ipopt} \\
&&&&&& Iter & local & coord. & \it total & Iter & deriv. & lin. & \it total & Objective & Violation & Objective & Violation\\ \hline  
case9241&21.4k&75.9k&1.7367&15.91\%& std.&55&2.92&1.55&\textbf{\emph{6.22}}&62&4.28&4.91&\it10.68&6243150&1.5E-08&6243090&3.9E-07\\
&&&&& api.&49&1.61&1.77&\textbf{\emph{ 5.09}}&101&7.27&9.97&\it19.11&7011200&2.7E-04&7011144&3.1E-07\\
&&&&& sad.&63&2.32&2.45&\textbf{\emph{ 7.07}}&75&5.51&6.83&\it13.94&6318469&2.9E-06&6318469&3.9E-07\\     
\addlinespace    
case9591&19.9k&76.5k&1.6594&26.01\%& std.&46&1.71&4.54&\textbf{\emph{8.36}}&52&3.60&10.48&\it15.76&1061684&3.1E-06&1061684&9.7E-08\\
&&&&& api.&56&2.19&5.92&\textbf{\emph{11.54}}&103&7.14&20.74&\it29.96&1570274&1.4E-06&1570264&2.2E-07\\
&&&&& sad.&57&2.18&4.90&\textbf{\emph{9.64}}&81&5.60&15.78&\it23.25&1167400&9.2E-06&1167401&9.9E-08\\
\addlinespace    
case10000&24k&69.6k&1.3193&16.10\%& std.&59&1.91&2.66&\textbf{\emph{6.08}}&78&4.62&6.65&\it12.69&1354031&8.7E-09&1354031&5.1E-07\\
&&&&& api.&61&1.88&2.79&\textbf{\emph{6.23}}&85&5.04&7.21&\it13.71&2678660&4.5E-07&2678659&1.1E-07\\
&&&&& sad.&62&1.81&2.86&\textbf{\emph{6.17}}&93&5.57&9.39&\it16.44&1490210&1.8E-07&1490210&4.8E-07\\                         
\addlinespace    
case10192&21.8k&81.6k&1.6695&23.68\%& std.&70&3.53&6.56&\textbf{\emph{13.37}}&50&4.05&8.84&\it14.51&1686923&4.7E-06&1686923&1.3E-07\\
&&&&& api.&57&2.36&5.40&\textbf{\emph{10.62}}&74&5.32&12.56&\it19.61&1977706&1.8E-05&1977686&2.5E-07\\
&&&&& sad.&57&2.44&5.12&\textbf{\emph{10.09}}&66&5.28&11.57&\it18.63&1720194&3.9E-05&1720194&1.3E-07\\        
\addlinespace    
case10480&22.5k&87.1k&1.7709&26.33\%  & std.&67&3.87&8.22&\textbf{\emph{16.21}}&63&5.24&14.77&\it22.09&2314648&2.6E-07&2375951&1.1E-07\\
 &&&&& api.&65&3.58&9.38&\textbf{\emph{17.41}}&66&5.17&15.46&\it22.66&2863484&1.9E-07&2924781&3.2E-07\\
 &&&&& sad.&66&3.90&8.50&\textbf{\emph{16.73}}&64&5.25&15.05&\it22.35&2314712&7.0E-07&2375951&1.1E-07\\
\addlinespace    
case13659&35.5k&102.4k&1.4984&9.65\% & std.&57&2.50&2.65&\textbf{\emph{8.65}}&180&15.96&26.80&\it45.94&8948202&7.5E-05&8948049&1.9E-07\\
&&&&& api.&72&3.99&3.12&\textbf{\emph{10.03}}&94&8.34&12.97&\it23.57&9385712&4.4E-04&9385711&2.5E-07\\
&&&&& sad.&62&2.74&2.49&\textbf{\emph{7.88}}&302&26.81&118.70&\it150.55&9042199&1.2E-05&9042198&3.2E-07\\
\addlinespace    
case19402&40.7k&162.3k&1.7887&19.48\%& std.&43&4.67&9.06&\textbf{\emph{17.66}}&69&10.28&33.83&\it48.14&1977815&1.6E-06&1977815&1.2E-07\\
&&&&& api.&64&6.73&15.62&\textbf{\emph{28.72}}&69&10.32&34.53&\it48.86&2583662&1.2E-07&2583663&4.9E-07\\
&&&&& sad.&56&5.87&11.31&\textbf{\emph{21.93}}&78&11.73&37.99&\it53.92&1983808&6.7E-06&1983809&1.2E-07\\
\addlinespace    
case20758&45.9k&162.3k&1.6063&13.35\%& std.&73&6.99&10.62&\it 24.79&43&6.17&12.55&\textbf{\emph{21.69}}&2618662&3.1E-05&2618636&1.4E-07\\
&&&&& api.&63&5.45&7.65&\textbf{\emph{17.48}}&69&9.84&18.90&\it32.15&3126508&5.8E-05&3126508&1.6E-07\\
&&&&& sad.&70&5.91&7.44&\textbf{\emph{18.20}}&55&7.83&15.50&\it26.48&2638220&4.1E-06&2638220&1.4E-07\\
\addlinespace    
case24464&52.1k&186.7&1.5458&11.97\%& std.&53&5.60&4.97&\textbf{\emph{13.05}}&57&9.48&21.04&\it34.17&2629531&2.1E-05&2629531&6.9E-08\\
&&&&& api.&57&7.86&5.66&\textbf{\emph{17.04}}&75&13.53&29.24&\it46.90&2684051&2.8E-05&2683962&3.2E-07\\
&&&&& sad.&64&6.39&7.22&\textbf{\emph{18.60}}&68&11.40&25.56&\it40.83&2653958&9.9E-06&2653958&7.1E-08\\
\addlinespace    
case30000&67.1k&196.2k&1.1798&8.36\%& std.&103&10.95&8.68&\textbf{\emph{27.05}}&128&20.85&37.79&\it63.38&1142458&1.6E-04&1142331&4.6E-07\\
&&&&& api.&100&11.91&10.13&\textbf{\emph{30.80}}&147&23.67&45.18&\it73.96&1778059&4.5E-04&1777931&1.8E-07\\
&&&&& sad.&212&21.97&29.77&\textbf{\emph{91.59}}&221&36.44&83.31&\it126.23&1317386&1.3E-04&1309979&2.0E-07\\
\addlinespace    
case78484&170.5k&613.5k&1.6057&6.37\%& std.&75&47.63&22.74&\textbf{\emph{80.35}}&95&52.10&131.07&\it199.05&15316174&2.1E-05&15315886&1.3E-07\\
&&&&& api.&74&47.21&23.22&\textbf{\emph{83.01}}&217&109.66&1688.24&\it1808.58&16140687&3.3E-07&19379770&1.3E+01\\
&&&&& sad.&77&49.79&24.17&\textbf{\emph{83.92}}&96&53.68&132.43&\it201.55&15316174&1.7E-06&15315886&1.3E-07\\
\addlinespace    
case21k&47.6k&113.2k&1.1858&8.41\%&std.&29&1.70&1.00&\textbf{\emph{3.30}}&63&5.90&8.22&\it16.08&2592246&8.4E-06&2592098&1.2E-07\\
\addlinespace    
case42k&95.1k&226.5k&1.1858&4.94\%&std.&42&5.14&1.39&\textbf{\emph{7.80}}&66&12.55&17.53&\it34.08&2592459&2.2E-06&2592458&5.8E-07\\
\addlinespace    
case99k&224.2k&522.9k&1.1858&2.11\%&std.&45&15.18&0.60&\textbf{\emph{16.54}}&67&30.06&46.45&\it86.50&2594077&2.2E-06&2594077&5.6E-08\\
\addlinespace    
case193k&434.8k&1035.5k&1.1857&0.97\%&std.&46&31.31&0.69&\textbf{\emph{37.35}}&71&63.30&93.94&\it178.33&2595600&8.0E-07&2595599&4.1E-08\\
\bottomrule
\end{tabular}\vspace{3pt}
{   \parbox{6.8 in}{There are three configurations for test cases from PGLib-\acrshort{opf}: Standard (STD), representing nominal operating conditions; Active Power Increase (API), simulating heavily loaded scenarios; and Small Angle Difference (SAD), enforcing strict bounds on voltage phase angle difference.}\par}
\end{table*}
\begin{table*}[htbp!]
\caption{Theoretical \& Practical Iterative Communication Effort}\label{TB::communication}
\centering
\footnotesize
\sisetup{separate-uncertainty, table-format=1.4(5), detect-all}
\setlength{\tabcolsep}{5pt}
\noindent    
\renewcommand{\arraystretch}{1.4}
\begin{tabular}{lK{1.5cm}K{1.5cm}K{3.5cm}K{4.5cm}}\toprule
&$\bar{\xi}$ & \acrshort{admm} & \acrshort{aladin} & \acrshort{baladin}\\\toprule
Forward  &$\begin{aligned}\sum_{\ell\in\mathcal{R}}\frac{\xi_\ell}{N^\text{reg}}\end{aligned}$ & $\begin{aligned}\sum_{\ell\in\mathcal{R}}2 N^x_\ell\end{aligned}$ & $\begin{aligned}\sum_{\ell\in\mathcal{R}}\, (N^x_\ell)^2 + \frac{3+2\xi_\ell}{2} N^x_\ell\end{aligned}$ & $\begin{aligned}\sum_{\ell\in\mathcal{R}}\, \frac{\xi^2_\ell}{2} (N^x_\ell)^2 + \frac{5\xi_\ell}{2} N^x_\ell\end{aligned}$\\ 
Backward  &$\begin{aligned}\sum_{\ell\in\mathcal{R}}\frac{\xi_\ell}{N^\text{reg}}\end{aligned}$& $\begin{aligned}\sum_{\ell\in\mathcal{R}} N^x_\ell\end{aligned}$ &$\begin{aligned}\sum_{\ell\in\mathcal{R}}\, (1+\xi_\ell)N^x_\ell \end{aligned}$ & $\begin{aligned}2N^\text{reg}+\sum_{\ell\in\mathcal{R}}\, \xi_\ell N^x_\ell \end{aligned}$\\
Total &$\begin{aligned}\sum_{\ell\in\mathcal{R}}\frac{\xi_\ell}{N^\text{reg}}\end{aligned}$& $\begin{aligned}\sum_{\ell\in\mathcal{R}}3 N^x_\ell\end{aligned}$ & $\begin{aligned}\sum_{\ell\in\mathcal{R}}\, (N^x_\ell)^2 + \frac{5+4\xi_\ell}{2} N^x_\ell\end{aligned}$ & $\begin{aligned}2N^\text{reg}+\sum_{\ell\in\mathcal{R}}\, \frac{\xi^2_\ell}{2} (N^x_\ell)^2 + \frac{7\xi_\ell}{2} N^x_\ell\end{aligned}$\\\hline 
case13659& 9.64\% &0.427 MB & 2.388 MB& 0.682 MB  \\
case24464& 11.97\%&0.635 MB& 4.276 MB& 2.104 MB \\
case78484& 6.37\%&2.016 MB& 13.731 MB& 5.791 MB  \\
case193k& 0.97\% &4.999  MB& 20.226 MB& 0.888 MB\\\bottomrule
\end{tabular}\vspace{3pt}\\
{   \parbox{4.8in}{Note: Theoretical worst-case communication efforts are expressed in terms of floating-point numbers, while practical communication efforts for large-scale AC \acrshort{opf} benchmarks are measured in MegaBytes, assuming the use of single-precision floating-point}\par}
\end{table*}

\subsection{Centralized vs. Distributed}\label{sec::comparison}

This section evaluates the performance of the proposed distributed approach against the state-of-the-art centralized nonlinear solver, \ipopt~\cite{wachter2006ipopt}. {Both centralized and distributed approaches use \mab~\cite{duff2004ma57} as the linear solver and \casadi~\cite{andersson2019casadi} for \acrfull{ad} to ensure a fair comparison, and both run on the same workstation. The centralized baseline benefits from intra-node parallelism in derivative evaluations and in the linear algebra pipeline of the KKT factorization/solve, whereas the distributed approach gains parallelism primarily through graph decomposition and distributed coordination (BALADIN), with one \matlab worker (one core) assigned to each region.} 

The evaluation involves the largest test cases from the PGLib-\acrshort{opf} benchmark\footnote{{We exclude the RTE cases due to a potential topology inconsistency that can cause solver-dependent convergence issues and confound runtime comparisons; see \url{https://github.com/power-grid-lib/pglib-opf/issues/48}.}}~\cite{pglibopf2019power}, each with three operation modes: \acrfull{std}, \acrfull{api}, and \acrfull{sad}. Additionally, the four largest test cases from recent studies~\cite{kardovs2022beltistos} are included. All these power grids are divided into 40 regions for analysis.

Detailed results in Table~\ref{TB::numerical-result} show that all the proposed distributed approaches converge to local optimizers\footnote{{Note that Table~\ref{TB::numerical-result} is generated from a fixed benchmark script, whereas Table~\ref{TB::Nregion} is produced by a separate script that varies $\abs{\mathcal{R}}$. }}. The optimality gaps between the centralized and distributed approaches are minimal, typically from $10^{-5}$ to $10^{-7}$. {We note that in a subset of cases, the reported residuals can be larger than IPOPT’s, even when the objective values are already close. This is consistent with our stopping criteria: we additionally apply a stagnation safeguard that terminates when the relative objective decrease becomes negligible, to avoid excessive iterations with little objective improvement.}  Notably, the case78484, with increased active power (API) using the centralized approach, fails to converge, underscoring a potential scalability advantage of distributed approaches.
\begin{figure}[htbp!]
\centering
\includegraphics[width=\linewidth]{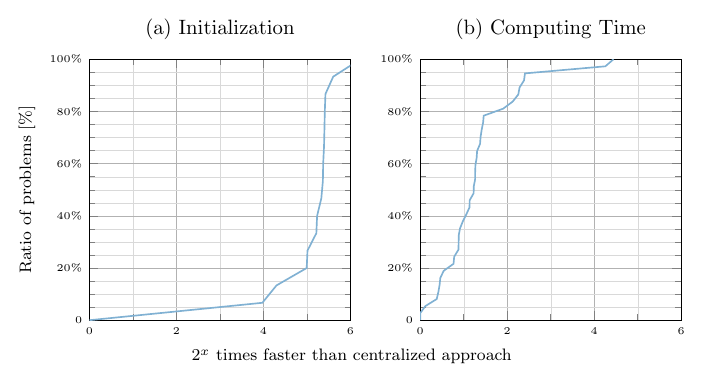}
\caption{Performance profile comparing Algorithm~\ref{alg::distIP} with IPOPT on large-scale AC OPF benchmarks (over 9,000 buses)}\label{fig::comparison::centralized} 
\end{figure}

Figure~\ref{fig::comparison::centralized} compares the performance efficiency across all these large-scale test cases. 
The distributed implementation of \acrshort{ad}, leveraging the \acrshort{spmd} model, significantly enhances the initialization speed. As shown in Fig.~\ref{fig::comparison::centralized}~(a), initialization is at least 16 times faster in more than $90\%$ of the cases compared to the centralized approach. Regarding algorithm efficiency, the distributed algorithm is at least twice as fast as the centralized method in over $60\%$ of the cases, with around $20\%$ of cases achieving at least four times the speed, as shown in Fig.~\ref{fig::comparison::centralized}~(b). It suggests that the larger the test case, the greater the computational benefits offered by the distributed approach, highlighting its significant advantages for handling large-scale problems.

\subsection{Impact of Network Sparsity}
\label{sec::effect::sparsity}
{
Power grids and other geographically grounded networks, such as transportation systems, are typically sparse: each node connects to only a few neighbors. This is reflected by the network density $\zeta$ (average number of lines per bus): for the large-scale benchmarks in Table~\ref{TB::numerical-result}, $\zeta$ ranges from 1.18 to 1.78. By contrast, a fully meshed network would require a direct connection between every pair of buses, corresponding to $\zeta=\zeta_{\max}=(N^{\mathrm{bus}}-1)/2$ (cf.~\eqref{eq::network::density}). For example, this would imply a transmission line between every city on the U.S.\ East Coast and every city on the West Coast. Such a topology is not representative of real power systems.}

{Since the convergence analyses in Sections~\ref{sec::global::convergence} and~\ref{sec::local::convergence} do not rely on network sparsity, the proposed distributed approach and the optimal graph-based decomposition criteria also apply to more highly meshed networks. However, as coupling becomes denser, region interfaces grow, and the coordination step becomes more expensive. The overall computing time can then be dominated by derivative condensation, coordination solves, and communication. In this regime, the distributed method may offer limited benefit and can even be slower than a centralized solver.}

\section{Iterative Communication Analysis}\label{sec::communication}
This section analyzes the communication requirements of Algorithm~\ref{alg::distIP}. Building upon the foundations in \cite{Engelmann2019}, we extend the theoretical analysis from specific AC \acrshort{opf} problems to general \acrshort{nlp} formulations. We further evaluate practical performance using the large-scale benchmarks introduced in previous sections.
\subsection{Worst-Case Communication Complexity}
To establish a baseline for comparison, we evaluate communication effort under a worst-case scenario. We assume that:
\begin{enumerate}
\item The number of active constraints is equal to the number of state variables $N^x_\ell$ for each region $\ell \in \mathcal{R}$ such that \acrshort{licq} still holds,
\item All matrices are treated as dense, neglecting potential sparsity patterns,
\item The analysis considers only the full Newton step, excluding additional communication required for inertia correction (Algorithm~\ref{alg::distributed::correction}) or globalization strategies (Algorithm~\ref{alg::basicGlobal}).
\end{enumerate}
These assumptions ensure a fair comparison across distributed algorithms, independent of problem-specific features such as sparsity patterns and complexity, while the practical evaluation is based on the four large-scale AC OPF benchmarks introduced in Section~\ref{sec::distIP::test}.

For the proposed distributed Algorithm~\ref{alg::distIP}, the coordinator solves the condensed \acrshort{qp} subproblem~\eqref{eq::coordination} with $W = \sum_{\ell\in\mathcal{R}} W_\ell$ and $h = -b + \sum_{\ell\in\mathcal{R}} h_\ell$. The communication effort for this is primarily associated with transferring the dual Hessian $W_\ell$ and the dual gradient $g_\ell$. Note that the coordinator is also responsible for reducing the barrier parameter $\mu$ (Step~\ref{alg::distIP::mu}) before solving the condensed~\acrshort{qp}. Since $h_\ell$ can be viewed as linear mapping of $\mu$, i.e.,
$h_\ell = h_{\ell,0} + \mu \;h_{\ell,\mu}$, the communication effort on $h_\ell$ would be doubled. Similar to the discussion earlier in this section, only $N^\text{clp}_\ell$ elements from $h_\ell$ should be transferred to a centralized coordinator. 

As a result, the forward communication in terms of floating-point numbers includes
\begin{align*}
\sum_{\ell\in\mathcal{R}}\;\underbrace{2N^\text{cpl}_\ell}_{h_\ell}  + \underbrace{\frac{N_\ell^\text{cpl}(N_\ell^\text{cpl}+1)}{2}}_{W_\ell}.
\end{align*}

Once the condensed linear system~\eqref{eq::coordination} is solved, the proposed algorithm only requires the dual step~$\Delta \lambda$  to be communicated during the backward phase, along with additional data for synchronizing the primal-dual steplength:
\begin{align*}
\underbrace{N^\text{reg}}_{\text{Eq.~\eqref{step::distIP::coordinate}}} + \sum_{\ell\in\mathcal{R}}\; \underbrace{N_\ell^\text{cpl}}_{\Delta\lambda}.
\end{align*}
\subsection{Practical Communication on AC OPF Benchmarks}
Recall that $\xi_\ell = \frac{N^\mathrm{cpl}_\ell}{N^x_\ell}$ measures the fraction of coupling variables in region $\ell\in\mathcal{R}$, and
$\bar{\xi} = \frac{1}{N^\mathrm{reg}}\sum_{\ell\in\mathcal{R}}\xi_\ell$ is the average coupling density across regions.

In large-scale power system applications, the decomposition typically yields a moderate number of regions compared with the total number of state variables, i.e.,
\[
N^\mathrm{reg}\ll \sum_{\ell\in\mathcal{R}}N^x_\ell.
\]
This regime is favorable for the proposed method: derivative condensing aggregates local second-order information before communication, so the coordination step exchanges information that scales with the local coupling variables rather than with the full problem dimension. As a result, \acrshort{baladin} substantially reduces communication relative to standard \acrshort{aladin}, even when the system is not weakly coupled.

A purely dense-matrix communication count is convenient for analysis but can be misleading in practice, since the Hessian and Jacobian in AC \acrshort{opf} are sparse. This particularly affects standard \acrshort{aladin}, whose coordination step solves a full-dimensional coupled \acrshort{qp} in the form~\eqref{eq::coupled::barrier::newton} and therefore benefits strongly from sparsity. For a fair comparison under a realistic structure, we report empirical communication measured in single-precision floats on large-scale AC \acrshort{opf} instances, summarized in Table~\ref{TB::communication}.

{
Except for the final test case, the per-iteration communication grows roughly linearly with system size for \acrshort{admm} and more steeply for \acrshort{aladin}, whereas \acrshort{baladin} shows intermediate scaling. This behavior is explained by the derivative-condensing step, which reduces the amount of information exchanged per iteration compared with standard \acrshort{aladin} (Table~\ref{TB::communication}). Moreover, the communication cost of \acrshort{baladin} depends on $\bar{\xi}$: when the partitioning keeps interfaces sparse (small $\bar{\xi}$), \acrshort{baladin} can match, and in some cases outperform, \acrshort{admm} in per-iteration communication.}

\subsection{Communication-Computation Trade-offs}\label{sec::trade-offs}
{It is also important to distinguish per-iteration communication from total communication to reach a given accuracy. Gradient-type methods such as \acrshort{admm} often communicate less per iteration, but they typically require many more iterations to reach comparable accuracy on nonlinear nonconvex problems such as AC \acrshort{opf}~\cite{muhlpfordt2021distributed,lanza2024ADMM_ALADIN,Dai2025thesis}. Thus, smaller per-iteration messages do not necessarily translate into lower end-to-end communication. In addition, because \acrshort{admm} relies on many inexpensive iterations, it requires more frequent synchronization and can be more sensitive to latency. In contrast, \acrshort{baladin} performs more work and exchanges richer information per iteration, but typically converges in far fewer iterations, which can reduce total communication and wall-clock time when latency is non-negligible.}

{These results also clarify scalability limits in extreme-scale and real-time settings. For million-bus systems or tight operational time budgets, the communication-computation trade-off becomes increasingly important, and performance will depend on both the communication infrastructure and the partition quality (in particular, the sparsity of the interfaces captured by $\bar{\xi}$). Faster links and sparse interfaces favor second-order coordination methods such as \acrshort{baladin}, whereas limited bandwidth or high latency may favor simpler schemes such as \acrshort{admm}. Overall, the communication effort and total runtime reported in Section~\ref{sec::distIP::test} depend not only on system size, but also on partitioning choices and the resulting coupling between subproblems.}

\section{Conclusion \& Outlook}    \label{sec::conclusion}

This paper presents a distributed solution for large-scale nonconvex AC \acrshort{opf} with convergence guarantees and fast local convergence. Scalability is achieved by decomposing the network into regional subproblems and coordinating them through a Schur-complement-based condensation step, which enables parallel computation while limiting the information exchanged between regions and thereby supporting privacy. Since the underlying formulation consists of a separable objective with local nonlinear constraints and affine coupling, the framework naturally extends to other distributed nonlinear optimization and control problems. Extensive simulations on the largest available test cases and across multiple operating scenarios quantify the impact of coupling density and show that the proposed method can outperform the state-of-the-art centralized nonlinear solver \ipopt~\cite{wachter2006ipopt} on moderate hardware.

{
The current study focuses on a synchronous execution model to enable a clean analysis of convergence and communication-computation trade-offs. Several practical extensions remain open and are important for real deployments. First, distributed globalization routines, e.g., line search, remain challenging, as also noted in~\cite{dinh2013dual,Boris2016aladin}. In particular, Step~\ref{alg::global::dual} can be computationally demanding and communication intensive, and improving this step is a key direction for future work. Second, residual stagnation can be exacerbated by ill-conditioning in condensed systems. Related effects of condensation steps on numerical conditioning and final solution accuracy have been discussed in~\cite{gill1991solving,shin2024accelerating}. Improving robustness and conditioning is therefore another important direction. Third, communication overhead can dominate at extreme scale, motivating further reduction of exchanged information and tighter integration with high-performance computing (HPC) communication primitives. Third, real systems may face asynchronous updates, node dropouts, and cyber-physical constraints such as variable latency, packet loss, and security requirements. Promising directions include asynchronous variants with bounded delays, fault-tolerant coordination strategies, e.g., temporarily freezing missing regions or re-partitioning when a node fails, and secure aggregation of the condensed quantities exchanged with the coordinator. Finally, while our numerical study in this paper is centered on large-scale AC \acrshort{opf} as a representative graph-based nonconvex benchmark, extending to other applications in different domains is part of ongoing and future work.
}




\appendices
\section{Globalization Strategy}

For the barrier problems~\eqref{eq::DistProblem::barrier}, the globalization strategy~\cite[Alg.~3]{Boris2016aladin} is outlined in Algorithm~\ref{alg::aladin::standard::global} with the merit function defined as
\begin{align*}
\Phi(x,s) = \sum_{\ell\in\mathcal{R}} &f^\mu_\ell(x_\ell,s_\ell) + \varepsilon_1 \norm{\sum_{\ell\in\mathcal{R}}A_\ell x_\ell-b}_1 \\ &+\varepsilon_2 \sum_{\ell\in\mathcal{R}}\left\{ \norm{c^\text{I}_\ell(x_\ell)+s_\ell}_1 +\norm{c^\text{E}_\ell(x_\ell)}_1\right\},
\end{align*}
where the positive penalty pfarameters $\varepsilon_1,\;\varepsilon_2$ are assumed to be sufficiently large such that $\Phi$ is an exact penalty function for the barrier problems~\eqref{eq::DistProblem::barrier}.

\begin{algorithm}[htbp!]
  \caption{Globalization strategy adapted from \cite[Alg.~3]{Boris2016aladin} for the barrier problems~\eqref{eq::DistProblem::barrier}}\label{alg::aladin::standard::global}
\small
\textbf{Initialization:} set $\alpha_1= \alpha_2 =\alpha_3 =1$.
\begin{enumerate}[label={\bfseries\alph*)}]
\item If the iterates from Step~9 in Algorithm~\ref{alg::distIP} satisfies \label{alg::global::full}
\begin{equation}
\label{eq::Meritdescent}
\hspace{-5mm}\Phi(z,s)-\Phi(x^+,s^+)\geq \eta \Big(
\sum_{i=1}^{N}\frac{\rho}{2}\left\|x_\ell-z_\ell\right\|_{\Sigma_\ell}^2 + \bar{\lambda} \Big\|\sum_{i=1}^{N}A_\ell x_\ell - b\Big\|_1
\Big)
\end{equation}
with $x^+\hspace{-1pt}=\hspace{-1pt}x+\Delta x$, $s^+ \hspace{-1pt}=\hspace{-1pt} s+ \Delta s$, then return $\alpha_1\hspace{-1pt}=\hspace{-1pt} \alpha_2 \hspace{-1pt}=\hspace{-1pt}\alpha_3 =1$.

\item If the full step is not accepted, set $x^+ = x$ and $\lambda^+ = \lambda$. If inequality~\eqref{eq::Meritdescent} holds, return $\alpha_1 = 1$ and $\alpha_2=\alpha_3 = 0$.\label{alg::global::half}

\item If both \ref{alg::global::full} and \ref{alg::global::half} failed, set $x^+ = z$ and choose $\alpha_3 \in (0,1]$ by solving \label{alg::global::dual}
\begin{equation}\label{eq::alpha3}
\alpha_3^* = \text{arg} \max_{\alpha_3 \in(0,1]} V_\rho(z,\lambda + \alpha_3 (\lambda^\text{QP} - \lambda ))
\end{equation} 
\changes{where the decoupled dual function $V_\rho(\bar x, \lambda)$ is defined in~\eqref{eq::Vrho}}, then
return $\alpha_1 = \alpha_2 = 0 $ and $\alpha_3 = \alpha_3^*$.
\end{enumerate}
\label{alg::basicGlobal}
\end{algorithm}
{\small
\section*{Acknowledgment}
\footnotesize The authors thank Frederik Zahn and François Pacaud for proofreading and discussion. }


\bibliographystyle{IEEEtran}
\bibliography{root}

\ifCLASSOPTIONcaptionsoff
\newpage
\fi




\end{document}

%% file: macros.tex
\usepackage{graphicx}
\usepackage{amssymb}
\usepackage{amsmath,amsthm,amsfonts}
\usepackage{dsfont}
\usepackage{xspace}

\usepackage{mathrsfs}
\usepackage{bm}
\usepackage{cite}

\usepackage[table,svgnames]{xcolor}

\usepackage{booktabs}
\usepackage{siunitx}
\usepackage{makecell}

\definecolor{lightgray}{gray}{0.9}
\usepackage[colorlinks=true,
linkcolor =red,
urlcolor  = blue,
citecolor =blue]{hyperref}

\usepackage{pgfplots}

\usepackage{subcaption}

\usepackage{booktabs}
\usepackage{multirow}
\usepackage{url}
\usepackage{graphicx} 
\usepackage{setspace}
\usepackage{physics}
\usepackage{enumitem}

\usepackage{adjustbox}
\usepackage{array}
\usepackage[]{footmisc}
\usepackage{circledsteps}

\usepackage{algorithm,algorithmicx}
\PassOptionsToPackage{noend}{algpseudocode}
\usepackage{algpseudocode}

\newtheorem{lemma}{Lemma}
\newtheorem{theorem}{Theorem}

\newtheorem{assumption}{Assumption}
\newtheorem{definition}{Definition}
\newtheorem{remark}{Remark}
\newtheorem{corollary}{Corollary}
\usepackage[acronym,nowarn,hyperfirst=false]{glossaries}
\setkeys{glslink}{hyper=false}
\newacronym{opf}{OPF}{optimal power flow}
\newacronym{pf}{PF}{Power Flow}
\newacronym{qp}{\uppercase {qp}}{quadratic program}
\newacronym{nlp}{\uppercase {nlp}}{nonlinear programming}
\newacronym{rapidpf}{rapid\uppercase {pf}}{rapid prototyping for distributed Power Flow}
\newacronym{admm}{\uppercase {admm}}{Alternating Direction Method of Multipliers}
\newacronym{aladin}{\uppercase {aladin}}{Augmented Lagrangian-based Alternating Direction Inexact Newton method}
\newacronym{ocd}{\uppercase {ocd}}{Optimality Condition Decomposition}
\newacronym{app}{\uppercase {app}}{Auxiliary Problem Principle}
\newacronym{sqp}{\uppercase {sqp}}{sequential quadratic programming}
\newacronym{kahip}{KaHIP}{Karlsruhe High Quality Partitioning}
\newacronym{kaffpa}{KaFFPa}{Karlsruhe Fast Flow Partitioner}
\newacronym{soc}{SOC}{second-order correction}
\newacronym{pegase}{PEGASE}{Pan European Grid Advanced Simulation and State Estimation}
\newacronym{spmd}{SPMD}{single program multiple data}
\newacronym{ad}{AD}{automatic differentiation}
\newacronym{std}{STD}{standard}
\newacronym{api}{API}{active power increased}
\newacronym{sad}{SAD}{small angle difference}
\newacronym{ders}{DERs}{distributed energy resources}
\newacronym{itd}{ITD}{integrated transmission-distribution}
\newacronym{licq}{LICQ}{linear independence constraint qualification} 
\newacronym{sosc}{SOSC}{second order sufficient condition} 
\newacronym{scc}{SCC}{strict complementarity conditions} 
\newacronym{kkt}{KKT}{Karush–Kuhn–Tucker} 
\newacronym{bfm}{BFM}{branch flow model} 
\newacronym{bim}{BIM}{bus injection model}

\newcommand{\matlab}{\textsc{matlab}\xspace}

\newcommand{\casadi}{\textsc{c}as\textsc{ad}i\xspace}
\newcommand{\ipopt}{\textsc{ipopt}\xspace}
\newcommand{\mab}{\textsc{MA}57\xspace}
\newcommand{\maa}{\textsc{MA}27\xspace}
\newcommand{\mac}{\textsc{MA}77\xspace}
\newcommand{\mad}{\textsc{MA}87\xspace}
\newcommand{\mae}{\textsc{MA}97\xspace}

\newcommand{\baladin}{\textsc{baladin}\xspace}

\usepackage[acronym,nowarn,hyperfirst=false]{glossaries}
\setkeys{glslink}{hyper=false}
\newacronym{baladin}{BALADIN}{Barrier ALADIN} 

\newcolumntype{K}[1]{>{\raggedleft\arraybackslash}p{#1}}

\newtheoremstyle{note}
{0pt}
{0pt}
{}
{}
{\itshape}
{:}
{.5em}
{}